\documentclass[11pt,a4paper]{article}
\pdfoutput=1
\usepackage[margin=1in]{geometry}
\usepackage[utf8]{inputenc}
\usepackage[dvipsnames]{xcolor}
\usepackage[bookmarksnumbered,linktocpage,hypertexnames=false,colorlinks=true,linkcolor=NavyBlue,urlcolor=NavyBlue,citecolor=ForestGreen,anchorcolor=green,breaklinks=true,pagebackref=true,pdfusetitle]{hyperref}
\usepackage{bbm,braket,microtype,amsmath,amssymb,amsthm,amsfonts,dsfont,dynkin-diagrams,mathtools,colortbl,booktabs,algorithm,algpseudocode,graphicx,enumitem,xspace,float,mathdots,ellipsis,caption,subcaption,mleftright,multirow,adjustbox}
\usepackage[nottoc]{tocbibind}
\usepackage[T1]{fontenc}
\usepackage{textcomp}
\usepackage[english]{babel}
\usepackage[capitalize,nameinlink]{cleveref}
\usepackage{tikz}
\usepackage{authblk}

\newtheorem*{theorem*}{Theorem}
\newtheorem{theorem}{Theorem}[section]\crefname{theorem}{Theorem}{Theorems}
\newtheorem{lemma}[theorem]{Lemma}\crefname{lemma}{Lemma}{Lemmas}
\crefname{claim}{Claim}{Claims}
\crefname{proposition}{Proposition}{Propositions}
\crefname{observation}{Observation}{Observations}
\newtheorem{corollary}[theorem]{Corollary}\crefname{corollary}{Corollary}{Corollaries}
\crefname{conjecture}{Conjecture}{Conjecture}
\theoremstyle{definition}
\newtheorem{definition}[theorem]{Definition}\crefname{definition}{Definition}{Definitions}
\crefname{problem}{Problem}{Problems}
\newtheorem{remark}[theorem]{Remark}\crefname{remark}{Remark}{Remarks}
\crefname{example}{Example}{Examples}
\crefname{condition}{Condition}{Conditions}
\numberwithin{equation}{section}
\newtheorem*{corollary*}{Corollary}
\newtheorem*{proposition*}{Proposition}

\usepackage{aliascnt}
\makeatletter
\let\c@algorithm\relax 
\makeatother
\newaliascnt{algorithm}{theorem}

\makeatletter
\let\c@table\relax 
\makeatother
\newaliascnt{table}{theorem}

\AddToHook{env/lemma/begin}{\crefalias{theorem}{lemma}}
\AddToHook{env/claim/begin}{\crefalias{theorem}{claim}}
\AddToHook{env/proposition/begin}{\crefalias{theorem}{proposition}}
\AddToHook{env/observation/begin}{\crefalias{theorem}{observation}}
\AddToHook{env/corollary/begin}{\crefalias{theorem}{corollary}}
\AddToHook{env/conjecture/begin}{\crefalias{theorem}{conjecture}}
\AddToHook{env/definition/begin}{\crefalias{theorem}{definition}}
\AddToHook{env/problem/begin}{\crefalias{theorem}{problem}}
\AddToHook{env/remark/begin}{\crefalias{theorem}{remark}}
\AddToHook{env/condition/begin}{\crefalias{theorem}{condition}}
\AddToHook{env/equation/begin}{\crefalias{theorem}{equation}}

\AddToHook{env/algorithm/begin}{\crefalias{theorem}{algorithm}}
\AddToHook{env/table/begin}{\crefalias{theorem}{table}}

\DeclareMathOperator{\conv}{conv}
\DeclareMathOperator{\diag}{diag}

\DeclareMathOperator{\GL}{GL}
\DeclareMathOperator{\Herm}{Herm}
\DeclareMathOperator{\Lie}{Lie}

\DeclareMathOperator{\spec}{spec}

\DeclareMathOperator{\supp}{supp}

\DeclareMathOperator{\U}{U}

\newcommand{\R}{\mathbb{R}}
\newcommand{\C}{\mathbb{C}}

\def\<#1>{\left\langle\ignorespaces#1\unskip\right\rangle}

\newcommand{\ot}{\otimes}

\DeclarePairedDelimiter\abs{\lvert}{\rvert}
\DeclarePairedDelimiter\norm{\lVert}{\rVert}

\newcommand{\closure}{\overline}

\newcommand{\Proj}{\mathbb{P}}
\newcommand{\weylchamber}{\mathcal D}

\newcommand{\diff}{\textnormal{d}}
\newcommand{\gldim}{n}
\newcommand{\tensor}{T} %
\newcommand{\sep}{\mid} %
\newcommand{\lowertriangular}{\mathchoice
    {\tikz[scale=1]{\draw (0,0) -- (4mm,0mm) -- (0,4mm) -- cycle}}
    {\tikz[scale=1]{\draw (0,0) -- (4mm,0mm) -- (0,4mm) -- cycle}}
    {\hspace{1pt}\tikz[scale=0.40]{\draw (0,0) -- (4mm,0mm) -- (0,4mm) -- cycle}}
    {\hspace{0.5pt}\tikz[scale=0.33]{\draw (0,0) -- (4mm,0mm) -- (0,4mm) -- cycle}}
}
\newcommand{\uppertriangular}{\mathchoice
    {\tikz[scale=1]{\draw (0,4mm) -- (4mm,4mm) -- (4mm,0mm) -- cycle}}
    {\tikz[scale=1]{\draw (0,4mm) -- (4mm,4mm) -- (4mm,0mm) -- cycle}}
    {\tikz[scale=0.40]{\draw (0,4mm) -- (4mm,4mm) -- (4mm,0mm) -- cycle}}
    {\tikz[scale=0.33]{\draw (0,4mm) -- (4mm,4mm) -- (4mm,0mm) -- cycle}}
}

\newcommand{\G}{\textnormal{GL}}
\newcommand{\Sl}{\textnormal{SL}}
\newcommand{\K}{\textnormal{U}}
\newcommand{\D}{\textnormal{D}}

\newcommand{\momentmaptorus}{\mu_\D}
\newcommand{\momentmaptorusgeneral}{\mu_D}
\newcommand{\grayzero}{{\color{gray}0}}
\allowdisplaybreaks[4]
\newcommand{\nurmievT}{\mathsf{T}}

\usepackage[normalem]{ulem}

\usepackage{thm-restate}

\begin{document}

\title{Explicit non-free tensors}
\author[1,2]{Maxim van den Berg}
\author[3]{Matthias Christandl}
\author[1]{Vladimir Lysikov}
\author[3]{Harold Nieuwboer}
\author[1]{Michael Walter}
\author[2]{Jeroen Zuiddam}

\affil[1]{Ruhr University Bochum, Bochum, Germany}
\affil[2]{University of Amsterdam, Amsterdam, Netherlands}
\affil[3]{University of Copenhagen, Copenhagen, Denmark}

\date{}
\maketitle
\vspace{-2em}
\begin{abstract} 
Free tensors are tensors which, after a change of bases, have free support: any two distinct elements of its support differ in at least two coordinates. They play a distinguished role in the theory of bilinear complexity, in particular in Strassen's duality theory for asymptotic rank.
Within the context of quantum information theory, where tensors are interpreted as multiparticle quantum states, freeness %
corresponds to a type of multiparticle Schmidt decomposition. In particular, if a state is free in a given basis, the reduced density matrices are diagonal. 
Although generic tensors in~$\C^n \ot \C^n \ot \C^n$ are non-free for~$n \geq 4$ by parameter counting, no explicit non-free tensors were known until now.
We solve this hay in a haystack problem by constructing explicit tensors that are non-free for every~$n \geq 3$.
In particular, this establishes that non-free tensors exist in~$\C^3 \ot \C^3 \ot \C^3$, where they are not generic.

To establish non-freeness, we use results from geometric invariant theory and the theory of moment polytopes.
In particular, we show that if a tensor~$T$ is free, then there is a tensor~$S$ in the~$\mathrm{GL}$-orbit closure of~$T$, whose support is free and whose moment map image is the minimum-norm point of the moment polytope of~$T$.
This implies a reduction for checking non-freeness from arbitrary basis changes of~$T$ to unitary basis changes of~$S$.
The unitary equivariance of the moment map can then be combined with the fact that tensors with free support have diagonal moment map image, in order to further restrict the set of relevant basis changes.
\end{abstract}

\setcounter{tocdepth}{1}
\section{Introduction}
\label{section:non-free tensor}

Generalizing the notion of a matrix, a~$k$-tensor can be described as a multidimensional array of numbers with~$k$ indices $T = (T_{i_1, \dotsc, i_k})_{i_1, \dotsc, i_k}$, where the~$i_\ell$ range from~$1$ to some natural number~$n_\ell$.
They play a pivotal role in algebraic complexity theory (bilinear complexity), quantum information theory (quantum states), combinatorics (cap sets and related problems).
In this paper we will restrict our attention to tensors of order~$3$, as they are of primary relevance for bilinear complexity.

We call any subset $S \subseteq [n]^3$ \emph{free} if any two distinct elements of $S$ differ in at least two coordinates.
For any tensor $T \in \C^n \otimes \C^n \otimes \C^n$ let $\supp(T) \subseteq [n]^3$ denote its support, consisting of the multi-indices at which~$T$ is non-zero. 
We call $T$ \emph{free} if there exist $(g_1, g_2, g_3) \in \GL_n(\C) \times \GL_n(\C) \times \GL_n(\C)$ such that $\supp((g_1 \otimes g_2 \otimes g_3) T)$ is free; in other words, $T$ has free support after possibly changing the basis on each of the tensor factors.
In particular, the tensor is sparse in this basis, in that it has at most~$n^2$ non-zero entries.
The class of free tensors play a particularly important role in several parts of the theory of moment polytopes and asymptotic spectra \cite{franz2002, strassenKomplexitatUndGeometrie2005}, which we will elaborate on below.

Generic tensors are not free, as can be seen by a dimension argument \cite[Remark 4.19]{christandlUniversalPointsAsymptotic2021}.
In \cite{connerGeometricApproachStrassen2021} the dimension of the Zariski-closure of the set of free tensors in $\C^n \ot \C^n \ot \C^n$ was determined to be $4n^2 - 3n$.
This implies that for every $n\geq4$ non-free tensors exist in $\C^n \ot \C^n \ot \C^n$, but they are dense for~$n=3$.
However, no explicit non-free tensors were known until now.

\subsubsection*{Main results}
In this paper, we resolve this ``hay in a haystack'' problem, in the sense that we give explicit non-free tensors in all cubic shapes of dimension at least three:
\begin{theorem}
    \label{thm:explicit non-free in balanced format}
For every $\gldim \geq 3$ the tensor $T \in \C^n \ot \C^n \ot \C^n$ given by
\[
T = \sum_{k=1}^{n-1} e_1 \otimes e_n \otimes e_k + \sum_{i=2}^n e_i \otimes e_{n+1-i}\otimes e_{n+1-i}+\sum_{i=1}^{n-1}e_i \otimes e_{n-i}\otimes e_n
\]
is non-free.
    In fact, a generic tensor with support contained in 
    \[
    \{(i,j,k) \in [n]^3 \mid \textnormal{$k < \gldim$ and $i+j=n+1$, or $k = \gldim$ and $i+j = \gldim$}\}
    \]
    is equivalent to the above tensor $T$ and hence non-free.
\end{theorem}

For example, for~$n=3$ we have the non-free tensor
\begin{equation*}
  T =
  \left[\!\arraycolsep=3pt
    \begin{array}{ccc|ccc|ccc}
      \grayzero&\grayzero&\grayzero& \grayzero&\grayzero&1& 1&\grayzero&\grayzero \\
      \grayzero&\grayzero&1& \grayzero&1&\grayzero& \grayzero&\grayzero&\grayzero \\
      1&1&\grayzero& \grayzero&\grayzero&\grayzero& \grayzero&\grayzero&\grayzero
    \end{array}\!
  \right],
\end{equation*}
and for $n = 4$ we have
\begin{equation*}
  T =
  \left[\!\arraycolsep=3pt
    \begin{array}{cccc|cccc|cccc|cccc}
      \grayzero&\grayzero&\grayzero&\grayzero& \grayzero&\grayzero&\grayzero&\grayzero& \grayzero&\grayzero&\grayzero&1& 1&\grayzero&\grayzero&\grayzero \\
      \grayzero&\grayzero&\grayzero&\grayzero& \grayzero&\grayzero&\grayzero&1& \grayzero&1&\grayzero&\grayzero& \grayzero&\grayzero&\grayzero&\grayzero \\
      \grayzero&\grayzero&\grayzero&1& \grayzero&\grayzero&1&\grayzero& \grayzero&\grayzero&\grayzero&\grayzero& \grayzero&\grayzero&\grayzero&\grayzero \\
      1&1&1&\grayzero& \grayzero&\grayzero&\grayzero&\grayzero& \grayzero&\grayzero&\grayzero&\grayzero& \grayzero&\grayzero&\grayzero&\grayzero
    \end{array}\!
  \right].
\end{equation*}
Freeness of the support in terms of such a slice representation means that (i) every row in every slice contains at most~$1$ non-zero entry, (ii) every column in every slice contains at most~$1$ non-zero entry, and (iii) distinct slices have non-overlapping support. 

We note that the non-free $T$ is close to a free tensor, by a rank-one perturbation: removing the last row of every slice of $T$, that is, removing the term $\sum_{k=1}^{n-1} e_1 \otimes e_n \otimes e_k$, gives a free tensor.

We also note that the tensor $T$ from \autoref{thm:explicit non-free in balanced format} may after relabeling of basis vectors be written as
\[
\sum_{i = 1}^{n-1} \bigl(e_i \ot e_i \ot e_i + e_i \ot e_n \ot e_n\bigr) \,+\, \sum_{j = 1}^{n-1} e_n \ot e_j \ot e_{j+1}
\]
However, the expression for $T$ in \autoref{thm:explicit non-free in balanced format} will be more natural for our proof.

\begin{corollary}
There exist non-free tensors in $\C^3 \ot \C^3 \ot \C^3$.
\end{corollary}
In format~$3 \times 3 \times 3$, we in fact find two non-free concise tensors.
These tensors are two of the finitely many~$\Sl$-unstable tensors that exist in this format, as classified by Nurmiev~\cite{nurmievOrbitsInvariantsCubic2000}.
We number these tensors as $\nurmievT_2$ and $\nurmievT_5$ to match the numbering in Nurmiev's paper; the tensor~$\nurmievT_2$ is equivalent to the tensor shown above.
The discovery that they are non-free was in part due to computational efforts \cite{vandenBerg2025momentPolytopeAlgorithm}.
The tensor~$\nurmievT_2$ also inspired our construction of non-free tensors in larger balanced formats.

By contrast, in any strictly smaller format, every tensor is known to be free; this can also be deduced from Nurmiev's classification~\cite{nurmievOrbitsInvariantsCubic2000} of the~$\Sl$-unstable (in the sense of geometric invariant theory) tensors in~$\C^3 \ot \C^3 \ot \C^3$ (noting that non-concise tensors are automatically~$\Sl$-unstable, hence it includes a classification of all tensors of smaller formats).
This classification also shows more explicitly that the free tensors are dense in~$\C^3 \ot \C^3 \ot \C^3$, as there is an open dense family (``family~1'') consisting of free tensors.

Another consequence of Nurmiev's classification, and in particular additional knowledge of the degeneration structure between unstable tensors in this format presented in~\cite{nurmievClosuresNilpotentOrbits2000}, is the following:
\begin{theorem}
    \label{theorem:free-tensors-not-degen-closed}
    There exists a concise free tensor~$T \in \C^3 \ot \C^3 \ot \C^3$
    such that its orbit closure $\overline{\G \cdot T}$ contains a concise non-free tensor.
\end{theorem}
The tensors here are~$T = \nurmievT_1$ and~$T' = \nurmievT_2$ from Nurmiev's classification of unstable tensors.
To the best of our knowledge, this phenomenon was previously unknown.

To show that the tensors described in~\cref{thm:explicit non-free in balanced format} are not free, we give a new ``reduction'' developed using tools from geometric invariant theory.
Of vital importance in this criterion is the~\emph{moment map}.
This is a map~$\mu\colon \C^n \otimes \C^n \otimes \C^n \setminus \{0\}\to \Herm_{n,n,n}$, taking a non-zero tensor and outputting a triple of Hermitian~$n \times n$ matrices.
We will elaborate further on the moment map below.
The points~$S$ in~$\overline{\G \cdot T} \setminus \{0\}$ for which~$\norm{\mu(S)}$ is minimal play an important role in analyzing (semi)stability in geometric invariant theory, and as we show below, they are also relevant to (non-)freeness of~$T$:
\begin{restatable}{theorem}{freereductionK}
    \label{thm:free-tensor-has-free-tensor-in-minimal-K-orbit}
    Suppose $T$ is a free tensor, and let $S \in \closure{\G \cdot T} \setminus \{0\}$ such that $\norm{\mu(S)}$ is minimal.
    Then there exists $k \in \K$ such that $k \cdot S$ has free support.
\end{restatable}
To prove this theorem we rely on classical results of geometric invariant theory, in particular results of Ness~\cite{nessStratificationNullCone1984}.
However, we also depend on more modern (analytically-minded) approaches, in particular the one taken by Georgoulas, Robbin and Salamon in~\cite{georgoulasMomentWeightInequalityHilbert2021}.
In particular we heavily rely on properties of the gradient flow of the Kempf--Ness function, combined with an observation regarding moment map evaluations for tensors with free support~\cite{sjamaarConvexityPropertiesMoment1998,franz2002}.

\subsubsection*{Motivation}
We now elaborate on the role of freeness in the theory of moment polytopes and asymptotic spectra \cite{franz2002, strassenKomplexitatUndGeometrie2005}.

We first make the link between freeness and the moment map explicit.
The moment map~$\mu\colon \C^n \ot \C^n \ot \C^n \setminus \{0\} \to \Herm_{n,n,n}$ has the following concrete description.
If~$T$ is normalized, then the~$(i,i')$-th entry of first component of~$\mu(T)$ records the Frobenius inner product between slices~$i$ and~$i'$ of~$T$. The second and third components of~$\mu$ are given by inner products of ``row-slices'' and ``column-slices'', respectively.
In particular the diagonal entries of each component of~$\mu(T)$ are simply given by the squared~$\ell_2$-norm of the corresponding slice, row or column of~$T$ (after normalizing~$T$).
Freeness of the support of a tensor implies that distinct slices/rows/columns have disjoint support, hence have zero inner product, leading to the following result~\cite{sjamaarConvexityPropertiesMoment1998,franz2002}:
\begin{proposition*}
If~$\supp(T)$ is free, then~$\mu([T])$ is a triple of~\emph{diagonal} matrices.
\end{proposition*}

Freeness of tensors plays a role in giving inner approximations of \emph{moment polytopes}~\cite{sjamaarConvexityPropertiesMoment1998,franz2002}.
The moment map is in fact a moment map in the symplectic geometry sense for the natural action of~$\K = \U_n \times \U_n \times \U_n$ on the projective space~$\Proj(\C^n \ot \C^n \ot \C^n)$.
For~$[T] \in \Proj(\C^n \ot \C^n \ot \C^n)$, the image of~$\mu$ on~$\overline{\G \cdot [T]} \subseteq \Proj(V)$ becomes a convex polytope in~$\R^n \times \R^n \times \R^n$ after restricting the image to diagonal matrices with weakly increasing diagonal entries.
This polytope is called the moment polytope of~$T$, and we will denote it by~$\Delta(T)$.
(We will more formally define the moment map and moment polytope in the main text.)

Franz used the notion of freeness to determine the~$3 \times 3 \times 3$ \emph{Kronecker polytope}: this is the largest possible moment polytope for tensors of shape $3 \times 3 \times 3$, and can alternatively be defined as the closure of all normalized triples~$\frac{1}{k} (\lambda,\mu,\nu)$ such that~$\lambda, \mu, \nu \vdash_3 k$ are partitions of~$k$ into at most~$3$ parts, with the condition that the Kronecker coefficient~$g_{\lambda,\mu,\nu} \neq 0$.
He established an alternative description of the moment polytope method, which lends itself well to obtaining outer approximations.
Next, he verified tightness of his outer approximation of the~$3 \times 3 \times 3$ Kronecker polytope by constructing all its vertices using tensors with free support.
Franz's method for obtaining outer approximations on moment polytopes is also highly relevant to our paper.
Sjamaar~\cite{sjamaarConvexityPropertiesMoment1998} considered moment polytopes in the general setting of a reductive group acting on a representation, and used a slightly more general definition (equivalent to freeness in the case of the tensor action) to provide inner approximations for the image of the moment map.
By acting with triples of invertible diagonal matrices on~$T$ (which does not affect the support), one can achieve the following lower bound on~$\Delta(T)$:
\begin{proposition*}[{\cite[Lem.~7.1]{sjamaarConvexityPropertiesMoment1998}}]
    If~$T$ has free support, then~$\Delta(T)$ contains~$p_{\scalebox{0.5}{$\searrow$}}$ for every~$p \in \conv \{ (e_i \sep e_j \sep e_k) \mid (i,j,k) \in \supp(T) \}$, where~$p_{\scalebox{0.5}{$\searrow$}}$ is obtained from~$p = (p_1 \sep p_2 \sep p_3)$ by putting each~$p_\ell$ in non-increasing order.
\end{proposition*}

For free tensors it is known that two classes of functions that are central in the theory of asymptotic spectra (Strassen's dual theory to asymptotic rank) the support functionals and quantum functionals, coincide \cite{christandlUniversalPointsAsymptotic2021}.
The quantum functionals are given by the maximum value of certain entropic quantities over the moment polytope of a tensor, whereas the support functionals are given by optimizing the same entropic quantities over all distributions on supports that a tensor can attain upon basis changes.
It is not known whether the quantum and support functionals agree for non-free tensors.
Better understanding of the class of non-free tensors can help settle this question, and in particular explicitly constructed non-free tensors could potentially be used for separations.
We do note that the non-free tensors we construct do not appear to yield such a separation, based on numerical evidence for~$n=3,4$.

\subsubsection*{Organization of the paper}
\Cref{section:preliminaries} discusses preliminaries from geometric invariant theory.
In \cref{section:moment-polytope-criterion-freeness} we give a criterion for (non-)freeness in terms of the moment polytope. 
In \cref{section:explicit-non-free-333}, we give two examples (denoted $\nurmievT_2$ and $\nurmievT_5$) of concise tensors in $\C^3 \ot \C^3 \ot \C^3$ which are not free.
In \cref{section:explicit-non-free-nnn}, we generalize one of the two examples ($\nurmievT_2)$ to an explicit family of concise non-free tensors in $\C^\gldim \ot \C^\gldim \ot \C^\gldim$ for arbitrary~$\gldim \geq 3$.
In~\cref{section:general-moment-polytope-criterion-freeness} we state and prove a (non-)freeness criterion for general rational actions of connected reductive groups.

\section{Moment map and moment polytopes}
\label{section:preliminaries}

\subsubsection*{Notation}
For natural numbers~$n \geq 1$ we write~$[n] = \{1, \dotsc, n\}$.
We abbreviate~$\GL_n = \GL_n(\C)$.
We will work with tensors in~$\C^a \ot \C^b \ot \C^c$ for~$a,b,c \geq 1$. 
There are a number of natural groups acting on this vector space.
For us of primary importance are products of the general linear group $\GL_a \times \GL_b \times \GL_c$, which we will denote by~$\G$ when~$a,b,c$ are clear from context, and its subgroup~$\U \coloneqq \U_a \times \U_b \times \U_c \subseteq \GL$ consisting triples of unitary matrices.
The action of~$(A,B,C) \in \G$ on~$T \in \C^a \ot \C^b \ot \C^c$ is given by
\[
    (A,B,C) \cdot T \coloneqq (A \ot B \ot C) T.
\]
The space of Hermitian~$a \times a$ matrices is denoted by~$\Herm_a$, and we write 
\[
\Herm_{a,b,c} \coloneqq \Herm_a \times \Herm_b \times \Herm_c.
\]

\subsubsection*{Moment map and moment polytopes}
We now discuss the moment map associated to the action of~$\G$ on tensors.
We begin with the geometric characterization of the moment map (see also \cite{burgisser2018tensorScaling,burgisserTheoryNoncommutativeOptimization2019} for an alternative introduction).
Recall that we defined the subgroup of unitary matrices as $\K = \U_a \times \U_b \times \U_c \subseteq \G$.
We will write~$V = \C^a \ot \C^b \ot \C^c$, and endow it with the~$\K$-invariant inner product~$\langle T, S \rangle = \sum_{i \in [a], j \in [b], k \in [c]} \overline{T_{i,j,k}} S_{i,j,k}$ (note that our inner products are~$\C$-linear in the second argument, and conjugate-linear in the first argument).

\begin{definition}[Moment map for tensors]
  \label{def:moment-map-tensors}
  The moment map~$\mu\colon V \setminus \{0\} \to \Herm_{a,b,c}$ is defined as follows.
  Consider~$\tensor \in V \setminus \{0\}$ as a map~$T_1\colon (\C^{a})^* \to \C^{b} \ot \C^{c}$, and identify~$(\C^{a})^* \cong \C^{a}$ using the standard (complex) inner product, and similar for $\C^{b} \ot \C^{c}$.
  Then
  \begin{equation}
    \label{eq:tensor-moment-map}
    \mu_1(\tensor) = \frac{1}{\norm{T}^2} T_1^* T_1^{\phantom{*}} \in \C^{a \times a},
  \qquad\text{i.e.}\quad
    \mu_1(\tensor)_{i,\ell} = \frac{1}{\norm{T}^2} \sum_{j \in [b],k \in [c]} \overline{\tensor_{i,j,k}} \tensor_{\ell,j,k},
  \end{equation}
  and similarly for the other two components.
\end{definition}
\begin{remark}
  The moment map is invariant under non-zero scalar multiplication of the argument: for~$\lambda \in \C^{\times}$, $\mu(\lambda T) = \mu(T)$.
  Therefore~$\mu$ also induces a well-defined map~$\mu\colon \Proj(V) \to \Herm_{a,b,c}$ on the projective space~$\Proj(V)$ consisting of lines through the origin in~$V$. 
\end{remark}
\begin{remark}
  The moment map as defined above is a moment map in the sense of symplectic geometry for the~$\U$-action on~$\Proj(V)$, see e.g.~\cite{nessStratificationNullCone1984,burgisser2018alternatingMinimization}.
\end{remark}
\begin{lemma}%
    \label{lemma:tensor moment map equivariance}
    For every $(A,B,C) \in \K$, $\mu\big((A,B,C)\cdot T\big)
    = (A \mu_1(T) A^{-1}, B \mu_2(T) B^{-1}, C \mu_3(T) C^{-1})$.
\end{lemma}
\begin{proof}
    The first component of~$\mu$ is given by~$T_1^* T_1 / \norm{T}^2$.
    Clearly~$\norm{T}^2$ is~$\K$-invariant, and $((A,B,C) T)_1 = (B \ot C) T_1 A^*$, so that
    \[
        ((A,B,C) T)_1^* ((A,B,C) T)_1 = A T_1^* (B \ot C)^* (B \ot C) T_1 A^* = A T_1^* T_1 A^{-1}
    \]
    since~$A, B, C$ are unitary matrices.
    The proof is analogous for the other components of~$\mu$.
\end{proof}

We now state a connection between freeness of tensors and the moment map.
\begin{lemma}[{\cite[Lem.~7.1]{sjamaarConvexityPropertiesMoment1998}, \cite[Prop.~2.2]{franz2002}}]
  \label{lem:free-implies-diagonal-moment-map}
  Suppose $T \in V \setminus \{0\}$ is a tensor with free support.
  Then $\mu(T)$ is a tuple of diagonal matrices.
\end{lemma}
\begin{proof}
  We establish that~$\mu_1(T)$ is diagonal, the proof for~$\mu_2(T)$ and~$\mu_3(T)$ is analogous.
  Observe that for~$i,j \in [a]$, expanding~\cref{eq:tensor-moment-map} yields
  \[
    (\mu_1(T))_{i,j} = \frac{1}{\norm{T}^2} \sum_{k = 1}^b \sum_{\ell=1}^c \overline{T_{i,k,\ell}} T_{j,k,\ell}.
  \]
  If~$T$ has free support, then~$\overline{T_{i,k,\ell}} T_{j,k,\ell} = 0$ whenever~$i \neq j$, because then~$(i,k,\ell)$ and~$(j,k,\ell)$ differ in exactly one coordinate.
  This shows that~$(\mu_1(T))_{i,j} = 0$  whenever~$i \neq j$, implying~$\mu_1(T)$ is a diagonal matrix.
\end{proof}

The next object of interest is the \emph{moment polytope}.
Let $\weylchamber_n \coloneqq \big\{\lambda \in \R^{n} \mid \lambda_{1} \geq \lambda_{2} \geq \cdots \geq \lambda_{n}\big\}$ be the non-increasing vectors and let $\weylchamber \coloneqq \weylchamber_a \times \weylchamber_b \times \weylchamber_c$.
The set~$\weylchamber$ will be referred to as the \emph{positive Weyl chamber} of~$\GL$.
We will often identify~$\weylchamber$ with the space of triples of real non-increasingly ordered diagonal matrices.
\begin{definition}[Moment polytope]
  Let~$T \in V \setminus \{0\}$.
  The moment polytope~$\Delta(T)$ of~$T$ is defined by
  \[
    \Delta(T) \coloneqq \mu\big(\overline{\G \cdot [T]}\big) \cap \weylchamber.
  \]
  where~$[T] \in \Proj(V)$ denotes the projective equivalence class of~$T$.
  Here~$\overline{\G \cdot [T]}$ is the closure\footnote{The closure is the same irrespective of whether one works with the Euclidean or Zariski topology on~$\Proj(V)$, see e.g.~\cite[Lem.~3.1]{wallachGeometricInvariantTheory2017}.} of the~$\G$-orbit of~$[T]$.
\end{definition}
\begin{remark}
  Observe that for every~$T \in V \setminus \{0\}$, there exist~$(A,B,C) \in \K$ such that~$\mu((A,B,C) \cdot T) \in \weylchamber$; after all, by~\cref{lemma:tensor moment map equivariance}, it suffices to pick~$A,B,C$ that diagonalize~$\mu_1(T), \mu_2(T), \mu_3(T)$ respectively, such that the eigenvalues are sorted non-increasingly. 
  If one writes~$\spec(H)$ for the non-increasingly ordered spectrum of a Hermitian matrix~$H$, then we obtain
  \[
    \Delta(T) = \Big\{\, \big(\spec\mu_1([S]), \spec\mu_2([S]), \spec\mu_3([S]) \big) \ \big|\ [S] \in \overline{\GL \cdot [T]}\, \Big\}.
  \]
\end{remark}
\begin{remark}
  One can also work affinely instead of projectively, or take the closure after taking the image of~$\mu$, leading to
  \[
    \Delta(T) = \mu\big(\overline{\G \cdot T} \setminus \{0\}\big) \cap \weylchamber = \overline{\mu(\G \cdot T)} \cap \weylchamber.
  \]
\end{remark}

The following remarkable theorem justifies the name of~$\Delta(T)$: %
\begin{theorem}[{\cite{nessStratificationNullCone1984,brion1987momentMapImage,walterEntanglementPolytopes2013}}]
    $\Delta(T)$ is a bounded convex polytope with rational vertices.
\end{theorem}

The moment polytope also has a characterization in terms of representation theory \cite{nessStratificationNullCone1984,brion1987momentMapImage,christandlUniversalPointsAsymptotic2021,burgisser2018tensorScaling}. 
This characterization exhibits the moment polytope as the set of normalized highest weights of $\G$ whose associated irreducible representations occur in the coordinate ring of $\overline{\G \cdot T}$.
Below, we will discuss another description of the moment polytope, and we refer to e.g.\ \cite{vandenBerg2025momentPolytopeAlgorithm} for a detailed discussion of all three different perspectives.

\subsubsection*{Moment polytopes via supports}
We now consider an alternative description of the moment polytope due to~\cite{franz2002}.
It exhibits the moment polytope as the intersection of the convex hulls of the possible supports of the tensor attainable under the action of lower triangular matrices.

Let $\G_{\lowertriangular} \subseteq \G$ be the triples of lower triangular matrices and $\G_{\uppertriangular} \subseteq \G$ the triples of upper triangular matrices.
The indices for the entries of~$T$ can be identified with triples of standard basis vectors, by identifying~$(i,j,k) \in [a] \times [b] \times [c]$ with~$(e_i, e_j, e_k) \in \R^a \times \R^b \times \R^c$.
We will then write~$\conv\supp(T)$ for the convex hull of those~$(e_i, e_j, e_k)$ for which~$T_{i,j,k} \neq 0$.\footnote{In a more general setting, the support is the set of \emph{weights} of the representation for which the projection of~$T$ onto the associated weight space is non-zero. The weights and the moment polytope can be realized in the same space, namely the (dual of the) Lie algebra of a maximal algebraic torus in~$G$.}
We then have the following description of~$\Delta(T)$:
\begin{theorem}[{\cite{franz2002}}]\label{theorem:tensor moment polytopes franz}
For every $U \in \GL_{\uppertriangular}$, we have the inclusion
\[
  \Delta(T) \supseteq \bigcap_{L \in \GL_{\lowertriangular}} \conv \supp(LU \cdot T) \cap \weylchamber,
\]
with equality for generic $U \in \GL_{\uppertriangular}$.
\end{theorem}
This theorem provides a method for proving both upper and lower bounds on~$\Delta(T)$.
Indeed, to prove that~$\Delta(T) \subseteq P$ for~$P \subseteq \R^a \times \R^b \times \R^c$, it is sufficient (and even equivalent) to show that
\[
\bigcap_{L \in \GL_{\lowertriangular}} \conv \supp(LU \cdot T) \cap \weylchamber \subseteq P
\]
for~\emph{every}~$U \in \GL_{\uppertriangular}$.
In the other direction, for proving lower bounds~$P \subseteq \Delta(T)$, it suffices to obtain one~$U \in \GL_{\uppertriangular}$ such that~$P \subseteq \bigcap_{L \in \GL_{\lowertriangular}} \conv \supp(LU \cdot T)$.
For our construction of non-free tensors, the upper bound will be the relevant direction (see \cref{subsection:nonfree-tensor-polytope-inequality}).

\section{A moment polytope criterion for non-freeness}
\label{section:moment-polytope-criterion-freeness}

We will now prove~\cref{thm:free-tensor-has-free-tensor-in-minimal-K-orbit}. Recall that $\mu$ denotes the moment map (\cref{def:moment-map-tensors}), $\Delta(T) = \{\mu\big(T'\big) \mid T' \in \overline{\G \cdot T} \setminus\{0\} \} \cap \weylchamber$ the moment polytope, and $\K = \U_a \times \U_b \times \U_c$ the maximal compact subgroup of $\G$ consisting of triples of unitary matrices.
We restate the theorem here for convenience:
\freereductionK*
We will use \cref{thm:free-tensor-has-free-tensor-in-minimal-K-orbit} in the contrapositive: to show that a tensor $T$ is not free, we exhibit an $S$ as in the theorem, such that for no $k \in \K$ the support of $k \cdot S$ is free.
To verify that $\norm{\mu(S)} = \min_{p \in \Delta(T)} \norm{p}$, we use the following criterion.
\begin{lemma}
    \label{lem:ness-checking-minimality}
    Let $T \in V  \setminus \{0\}$. Then the following are equivalent:
    \begin{enumerate}[label=\upshape(\arabic*)]
        \item\label{item:ness minimality minimality} $\norm{\mu(T)} = \min_{p \in \Delta(T)} \norm{p}$,
        \item\label{item:ness minimality stabilizer} there exists some $\lambda \in \R$ such that $\exp(t \mu(T)) \cdot T = e^{\lambda t} T$ for all $t \in \R$.
    \end{enumerate}
\end{lemma}
\begin{proof}
The equivalence of (1) and (2) is shown in~\cite[Thm.~6.1]{nessStratificationNullCone1984}, noting that~$T$ is always~$\G$-unstable and~$\mu(T)$ is never zero (since every component has trace~$1$).
\end{proof}
In particular, condition~\ref{item:ness minimality stabilizer} is easily verified for an explicit tensor $T$.

We will now work towards the proof of~\cref{thm:free-tensor-has-free-tensor-in-minimal-K-orbit}.
The first ingredient is the following.
\begin{theorem}[{\cite[Thm.~6.4]{georgoulasMomentWeightInequalityHilbert2021}}]
    \label{thm:GRS-gradflow-ODE-properties-tensor}
    Let $T \in V \setminus \{0\}$.
    Then there is a solution~$T(t) \in V$, where $t \geq 0$, to the ordinary differential equation
    \begin{equation}
        \label{equation:ode}
        \frac{\diff}{\diff t} T(t) = - \bigl( \mu_1(T(t)) \ot I \ot I \,+\, I \ot \mu_2(T(t)) \ot I \,+\, I \ot I \ot \mu_3(T(t)) \bigr) T(t), \quad T(0) = T.
    \end{equation}
    It satisfies~$T(t) \in \G \cdot T$ for all $t \geq 0$.
    Moreover,
    writing~$[S] \in \Proj(V)$ for the equivalence class in projective space of a non-zero tensor~$S \in V \setminus \{0\}$,
    the limit $[T]_\infty := \lim_{t \to \infty} [T(t)] \in \Proj(V)$ exists, and $\norm{\mu([T]_\infty)} = \min_{p \in \Delta(T)} \norm{p}$.
\end{theorem}
\begin{remark}
  \label{remark:gradflow-ODE-explanation-infinitesimal}
  Consider the linear action of~$\C^{a\times a} \oplus \C^{b\times b} \oplus \C^{c \times c}$ on~$V$ defined by
  \[
    (A,B,C) \star T = 
    (A \otimes I \otimes I) T + 
    (I \otimes B \otimes I) T +
    (I \otimes I \otimes C) T.
  \]
  We will refer to this action as the \emph{infinitesimal action}, and it is the action of the Lie algebra of~$\G$ induced by the action of~$\G$ on~$V$ (and the moment map maps into the Lie algebra).
  With this notation, the ODE in \cref{thm:GRS-gradflow-ODE-properties-tensor} reads
  \[
    \frac{\diff}{\diff t} T(t) = - \mu(T(t)) \star T(t), \quad T(0) = T.
  \]
\end{remark}
\begin{remark}
  The original statement of~\cref{thm:GRS-gradflow-ODE-properties-tensor} in \cite{georgoulasMomentWeightInequalityHilbert2021} is rather more general, and in terms of the gradient flow of the function~$[T] \mapsto \frac12 \norm{\mu([T])}^2$ on~$\Proj(V)$.
  This gradient however is easily shown to be equal to the infinitesimal action of~$\mu([T])$ on~$[T]$, see e.g.~\cite[Lem.~6.1]{nessStratificationNullCone1984} or~\cite[Lem.~3.1]{georgoulasMomentWeightInequalityHilbert2021}.
\end{remark}

\begin{remark}
  A natural question is whether~$\lim_{t\to\infty} T(t) / \norm{T(t)}$ also converges, i.e., whether working projectively can be replaced by merely normalizing the tensor.
  This does not appear to be known, although it is known that~$\mu(T(t))$ converges, see~\cite[Thm.~4.13]{hiraiGradientDescentUnbounded2024}.
\end{remark}

We also use the following theorem, showing that the minimum-norm points in any $\GL$-orbit-closure lie in one $\K$-orbit. %
\begin{theorem}[{\cite[Thm.~6.5]{georgoulasMomentWeightInequalityHilbert2021}}]
    \label{thm:second-ness-uniqueness}
    Let $T \in V \setminus \{0\}$. Let~$S, S' \in \closure{\G \cdot T}$ such that~$\norm{\mu(S)} = \norm{\mu(S')} = \min_{p \in \Delta(T)} \norm{p}$.
    Then~$[S] \in \K \cdot [S']$.
\end{theorem}
This is an extension of a classical theorem of Ness~\cite[Thm.~6.2]{nessStratificationNullCone1984}, where it is additionally assumed that~$S, S' \in \G \cdot T$.
To prove that $\nurmievT_2$ and $\nurmievT_5$ are not free it suffices to use the original theorem of Ness. However, \Cref{thm:second-ness-uniqueness} will give a more general criterion for freeness of tensors that we will need later.

\begin{corollary}
    \label{cor:moment-map-norm-minimizers-unique-in-orbit-closure}
    Let $T \in V \setminus \{0\}$.
    If $S \in \closure{\G \cdot T}$ is such that $\norm{\mu(S)} = \min_{p \in \Delta(T)} \norm{p}$, then $[S] \in \K \cdot [T]_\infty$.
\end{corollary}
\begin{proof}
  This follows directly from~\cref{thm:GRS-gradflow-ODE-properties-tensor,thm:second-ness-uniqueness}.
\end{proof}

\begin{proof}[Proof of~\cref{thm:free-tensor-has-free-tensor-in-minimal-K-orbit}.]
    Assume without loss of generality that~$T$ has free support, otherwise replace it by a tensor in its orbit with free support; this does not affect~$\overline{\G \cdot T}$ or~$\Delta(T)$.
    Consider the solution $T(t)$ to the gradient flow and its projective limit $[T]_\infty$ from \cref{thm:GRS-gradflow-ODE-properties-tensor}.
    We claim that~$T(t)$ remains in the orbit of~$T$ under the action of the maximal algebraic torus~$\D = (\C^*)^a \times (\C^*)^b \times (\C^*)^c \subseteq \G$, consisting of all diagonal matrices in~$\G$.
    Since the support of a tensor is invariant under the action of $\D$, this implies that $T(t)$ has the same support as $T$ for any $t$.
    Hence if we write $[T]_\infty = [T']$, then the support of $T' \in \overline{\D \cdot T}$ is contained in the support of $T$, and hence $T'$ is also a free tensor.
    \Cref{cor:moment-map-norm-minimizers-unique-in-orbit-closure} implies that $[T'] \in \K \cdot [S]$.
    After appropriate rescaling of $T'$, which does not affect freeness of its support, this implies $T' \in \K \cdot S$, and we are done.
    
    We now prove the claim.
    We first define~$\momentmaptorus \colon V \setminus \{0\} \to \Herm_{a,b,c}$ to be the composition of $\mu$ and the componentwise projection onto the diagonal part of~$\mu(T)$, i.e., the~$(i,j)$-th entry of $(\momentmaptorus(T))_1$ is~$0$ if~$i \neq j$ and equal to~$(\mu_1(T))_{i,i}$ otherwise.
    Consider the ODE
    \[
      \frac{\diff}{\diff t} \tilde T(t) = - \momentmaptorus(\tilde T(t)) \star \tilde T(t), \quad \tilde T(0) = T.
    \]
    where~$\star$ is as defined in~\cref{remark:gradflow-ODE-explanation-infinitesimal}. 
    Then the solution $\tilde T(t)$ is in the $\D$-orbit of $T$.
    To see this, first observe that $t \mapsto \exp(\int_0^t -\momentmaptorus(\tilde T(s)) \, \mathrm{d}s) \cdot T$ is also a solution to the ODE, because
    \begin{align*}
        \frac{\diff}{\diff t} \left(\exp\left(\int_0^t -\momentmaptorus(\tilde T(s)) \, \mathrm{d}s\right) \cdot T\right) & = - \momentmaptorus(\tilde T(t)) \star \left(\exp\left(\int_0^t -\momentmaptorus(\tilde T(s)) \, \mathrm{d}s\right) \cdot T\right) \\ & = - \momentmaptorus(\tilde T(t)) \star \tilde T(t),
    \end{align*}
    where we used that $\int_0^t -\momentmaptorus(\tilde T(s)) \, \mathrm{d}s$ is a diagonal matrix for every~$t \geq 0$, so that its matrix exponential is easy to differentiate. 
    By uniqueness of solutions to ODEs\footnote{The moment map, its diagonal part, and the infinitesimal action are locally Lipschitz functions, so uniqueness follows from a standard application of Picard--Lindel\"of theory.} we find that 
    \begin{equation*}
      \tilde T(t) = \exp\left(\int_0^t -\momentmaptorus(\tilde T(s)) \, \mathrm{d}s\right) \cdot T.
    \end{equation*}
    Because $\exp\bigl(\int_0^t -\momentmaptorus(\tilde T(s)) \, \mathrm{d}s\bigr) $ is an invertible diagonal matrix for every $t \geq 0$, this implies that $\tilde T(t) \in \D \cdot T$.
    This yields~$\mu(\tilde T(t)) = \momentmaptorus(\tilde T(t))$ by~\cref{lem:free-implies-diagonal-moment-map}.
    Hence $\tilde T(t)$ is also a solution to the ODE of \cref{thm:GRS-gradflow-ODE-properties-tensor} (\cref{equation:ode}).
    By another application of uniqueness of solutions to ODEs, we find~$\tilde T(t) = T(t)$ for all~$t \geq 0$, and conclude~$T(t) = \tilde T(t) \in \D \cdot T$.
\end{proof}

\section{Explicit non-free \texorpdfstring{$3 \times 3 \times 3$}{3 x 3 x 3} tensors}
\label{section:explicit-non-free-333}

Let~$e_{i,j,k} \coloneqq e_i \ot e_j \ot e_k$, and consider the tensors
\begin{equation*}
    \nurmievT_2 = e_{1,2,3} + e_{1,3,2} + e_{2,1,3} + e_{2,2,1} + e_{2,2,2} + e_{3,1,1}
\end{equation*}
and
\begin{equation*}
    \nurmievT_5 = e_{1,1,3} + e_{1,3,1} + e_{1,3,2} + e_{2,2,1} + e_{3,1,2}
\end{equation*}
which are the second and fifth $\Sl$-unstable tensors from Nurmiev's classification~\cite{nurmievOrbitsInvariantsCubic2000}.
\begin{theorem}
    \label{thm:c333 t2 t5 not free}
    The tensors $\nurmievT_2$ and $\nurmievT_5$ are not free.
\end{theorem}
We prove that $\nurmievT_2$ is not free by giving an explicit $S \in \closure{\G \cdot T} \setminus \{0\}$ (in fact, even in $\G \cdot T$) such that $\mu([S])$ is the minimum-norm point of $\Delta(T)$, and subsequently showing that no element of $\K \cdot S$ has free support.
To show that no element of $\K \cdot [S]$ has free support, we exploit the $\K$-equivariance of the moment map, together with the fact that the moment map evaluates to a tuple of diagonal matrices if the tensor has free support (\cref{lem:free-implies-diagonal-moment-map}).
The same strategy works for~$\nurmievT_5$, for which we provide the essential data necessary to mimic the proof for~$\nurmievT_2$ in \cref{remark:T5 non-free proof}.

The next proposition records a computation yielding an element of $\K \cdot [\nurmievT_2]_\infty$ (with notation as in~\cref{thm:GRS-gradflow-ODE-properties-tensor}):
\begin{lemma}
    \label{prop:c333 t2 computation}
    Let
    \begin{equation*}
        g_1 =
        \begin{bmatrix}
            \sqrt{\frac{5}{2 \cdot 11}} & \grayzero & \grayzero \\
            \grayzero & 1 & \grayzero \\
            \grayzero & \grayzero & \sqrt{\frac{7 \cdot 11}{2 \cdot 5}}
        \end{bmatrix}, \,
        g_2 =
        \begin{bmatrix}
            \sqrt{\frac{5^2}{7 \cdot 11^2}} & \grayzero & \grayzero \\
            \grayzero & \sqrt{\frac{2 \cdot 5}{7 \cdot 11}} & \grayzero \\
            \grayzero & \grayzero & 1
        \end{bmatrix}, \,
        g_3 =
        \begin{bmatrix}
            1 & \grayzero & \grayzero \\
            -\frac{4}{\sqrt{3 \cdot 5 \cdot 7}} & \sqrt{\frac{11^2}{3 \cdot 5 \cdot 7}} & \grayzero \\
            \grayzero & \grayzero & \frac{11}{5}
        \end{bmatrix}.
    \end{equation*}
    Then~$g = (g_1, g_2, g_3)$ is such that $S_2 := g \cdot \nurmievT_2$ satisfies: 
    \begin{enumerate}[label=\upshape(\arabic*)]
    \item $S_2 =
        \sqrt{\tfrac{1}{7}}\, e_{1,2,3}
        + \sqrt{\tfrac{11}{42}}\, e_{1,3,2}
        + \sqrt{\tfrac{1}{7}}\, e_{2,1,3}
        + \sqrt{\tfrac{10}{77}}\, e_{2,2,1}
        + \sqrt{\tfrac{2}{33}}\, e_{2,2,2}
        + \sqrt{\tfrac{5}{22}}\, e_{3,1,1}
        - \sqrt{\tfrac{8}{231}}\, e_{3,1,2}$,
    \item $\mu(S_2) = \bigl(\diag(\tfrac{17}{42}, \tfrac{1}{3}, \tfrac{11}{42}), \diag(\tfrac{17}{42}, \tfrac{1}{3}, \tfrac{11}{42}), \diag(\tfrac{5}{14}, \tfrac{5}{14}, \tfrac{2}{7})\bigr)$,
    \item for every~$t \in \R$, $\exp(-t \mu(S_2)) \cdot S_2 = e^{- t \, 43/42} S_2$.
    \end{enumerate}
\end{lemma}
\begin{proof}
    This follows from a direct computation.
\end{proof}

\begin{lemma}
\label{prop:c333 s2 no free support}
There is no~$k \in \K$ such that~$k \cdot S_2$ has free support.
\end{lemma}
\begin{proof}
    The moment map is $\K$-equivariant (\cref{lemma:tensor moment map equivariance}). \Cref{prop:c333 t2 computation}(2) shows that $\mu(S_2)$ is a tuple of diagonal matrices, such that $\mu_1(S_2) = \mu_2(S_2) = \diag(\frac{17}{42}, \frac{1}{3}, \frac{11}{42})$ has no repeated eigenvalues.
    Suppose~$k = (k_1, k_2, k_3) \in \K$ is such that~$k \cdot S_2$ has free support.
    Then~$\mu_1(k \cdot S_2) = k_1 \mu_1 (S_2) k_1^{-1}$ is diagonal by~\cref{lem:free-implies-diagonal-moment-map}, but also isospectral to~$\mu_1(S_2)$.
    By further composing with a local permutation we may assume that~$\mu_1(k \cdot S_2) = \mu_1(S_2)$.
    Similarly we may assume that~$\mu_2(k \cdot S_2) = \mu_2(S_2)$.
    This implies that~$k_1$ and~$k_2$ are both diagonal unitary matrices, since the eigenvalues of~$\mu_1(S_2) = \mu_2(S_2)$ are all distinct.
    These do not affect the support of a tensor, and so $k \cdot S_2 = (k_1, k_2, I) \cdot ((I, I, k_3) \cdot S_2)$ has the same support as $(I, I, k_3) \cdot S_2$.
    Therefore it just remains to show that $(I, I, k_3) \cdot S_2$ does not have free support.
    Again we may assume without loss of generality that~$\mu_3(k \cdot S_2) = \mu_3(S_2)$ by acting with a permutation on the third system.
    Now $\mu_3(S_2) = \diag(\frac{5}{14}, \frac{5}{14}, \frac{2}{7})$, so $k_3$ is necessarily of the form
    \begin{equation}
        \label{eq:s2 k3 form}
        k_3 = \begin{bmatrix}
            U & 0 \\
            0 & z
        \end{bmatrix}
    \end{equation}
    for some $2 \times 2$ unitary $U$ and $z \in \C$ with $\abs{z} = 1$.
    The choice of $z$ does not affect the support, so we may assume $z = 1$.
    It remains to argue that there is no $U$ such that for $k_3$ as in \cref{eq:s2 k3 form} the tensor $(I, I, k_3) \cdot S_2$ has free support.
    For this, it is helpful to view $S_2$ in terms of its slices (\cref{prop:c333 t2 computation}):
    \begin{equation*}
        S_2 =
        \left[
        \begin{array}{ccc|ccc|ccc}
            \grayzero & \grayzero & \grayzero & \grayzero & \grayzero & \sqrt{\frac{1}{7}} & \sqrt{\frac{5}{22}} & - \sqrt{\frac{8}{231}} & \grayzero \\
            \grayzero & \grayzero & \sqrt{\frac{1}{7}} & \sqrt{\frac{10}{77}} & \sqrt{\frac{2}{33}} & \grayzero & \grayzero & \grayzero & \grayzero \\
            \grayzero & \sqrt{\frac{11}{42}} & \grayzero & \grayzero & \grayzero & \grayzero & \grayzero & \grayzero & \grayzero
        \end{array}
        \right].
    \end{equation*}
    The unitary $U$ acts on $S_2$ by taking linear combinations of the first two columns of each slice simultaneously. 
    Now we note that there is no $U$ that makes the first two columns free in each of these three slices simultaneously.
    The obstruction is that there is no~$2 \times 2$ unitary matrix~$U$ that the vectors
    \begin{equation*}
        U \begin{bmatrix} 0 \\ \sqrt{\frac{11}{42}} \end{bmatrix},
        \quad
        U \begin{bmatrix} \sqrt{\frac{10}{77}} \\ \sqrt{\frac{2}{33}} \end{bmatrix},
        \quad
        U \begin{bmatrix} \sqrt{\frac{5}{22}} \\ - \sqrt{\frac{8}{231}}\end{bmatrix}
    \end{equation*}
    each have at most one non-zero entry.
    If such a $U$ existed, two of the vectors must be parallel to one another by the pigeonhole principle, but this is clearly not the case.
\end{proof}
\begin{proof}[Proof of~\cref{thm:c333 t2 t5 not free} for $\nurmievT_2$.]
    We have $S_2 \in \GL \cdot \nurmievT_2$, and $\norm{\mu(S_2)} = \min_{p \in \Delta(T)} \norm{p}$ by \cref{prop:c333 t2 computation,lem:ness-checking-minimality}.
    Therefore, if $\nurmievT_2$ were free, \cref{thm:free-tensor-has-free-tensor-in-minimal-K-orbit} would show that $S_2$ contains some tensor with free support in its $\K$-orbit.
    However, \cref{prop:c333 s2 no free support} shows that this is not the case.
\end{proof}

\begin{remark}
\label{remark:T5 non-free proof}

The proof that $\nurmievT_5 = e_{1,1,3} + e_{1,3,1} + e_{1,3,2} + e_{2,2,1} + e_{3,1,2}$,
is not free is similar to the proof for $\nurmievT_2$ given above.
In its slice form, $\nurmievT_5$ is given by
\begin{equation*}
    \nurmievT_5 = \left[\!
    \begin{array}{ccc|ccc|ccc}\!
        \grayzero & \grayzero & 1  &  \grayzero & \grayzero & \grayzero  &   \grayzero & 1 & \grayzero \\
        \grayzero & \grayzero & \grayzero  &  1 & \grayzero & \grayzero  &   \grayzero & \grayzero & \grayzero \\
        1 & 1 & \grayzero  &  \grayzero & \grayzero & \grayzero  &   \grayzero & \grayzero & \grayzero
    \end{array}\!\right].
\end{equation*}
This support is not free as the first slice and third row contains two non-zero entries.
One should observe that this has a similar structure as $\nurmievT_2$, which in slice form is written as
\begin{equation*}
    \nurmievT_2 = \left[\!
    \begin{array}{ccc|ccc|ccc}
        \grayzero & \grayzero & \grayzero  &  \grayzero & \grayzero & 1  &   1 & \grayzero & \grayzero \\
        \grayzero & \grayzero & 1  &  1 & 1 & \grayzero  &   \grayzero & \grayzero & \grayzero \\
        \grayzero & 1 & \grayzero  &  \grayzero & \grayzero & \grayzero  &   \grayzero & \grayzero & \grayzero
    \end{array}\!\right].
\end{equation*}
The minimum-norm point in the moment polytope of $\nurmievT_5$ can be found as follows:
in its orbit lies the tensor
\begin{equation*}
  S_5 = \left[\!
     \begin{array}{ccc|ccc|ccc}
       \grayzero & \grayzero & \sqrt{\frac{1}{5}}                    &  \grayzero & \grayzero & \grayzero  &   \grayzero & \sqrt{\frac{7}{3\grayzero}} & \grayzero \\
       \grayzero & \grayzero & \grayzero                                     &  \sqrt{\frac{2}{7}} & - \sqrt{\frac{1}{21}} & \grayzero  &   \grayzero & \grayzero & \grayzero \\
       \sqrt{\frac{4}{35}} & \sqrt{\frac{5}{42}} & \grayzero &  \grayzero & \grayzero & \grayzero  &   \grayzero & \grayzero & \grayzero
    \end{array}\!\right],
\end{equation*}
and evaluating the moment map on this tensor gives
\begin{equation*}
    \mu(S_5) = \bigl(\diag(\tfrac{13}{30}, \tfrac{1}{3}, \tfrac{7}{30}), \diag(\tfrac{13}{30}, \tfrac{1}{3}, \tfrac{7}{30}), \diag(\tfrac{2}{5}, \tfrac{2}{5}, \tfrac{1}{5})\bigr),
\end{equation*}
for which~$\exp(t \mu(S_5)) S_5 \propto S_5$, establishing that~$\mu(S_5)$ is the minimum-norm point in $\Delta(S_5) = \Delta(\nurmievT_5)$ by~\cref{lem:ness-checking-minimality}.
The argument for proving that $S_5$ does not have any tensor with free support in its $\K$-orbit is exactly the same as for $\nurmievT_2$.
\end{remark}

\begin{remark}[Quantum and support functionals]
    Although $\nurmievT_2$ and $\nurmievT_5$ are not free, we observe numerically that the quantum functionals $F_\theta$~\cite{christandlUniversalPointsAsymptotic2021} and Strassen's (upper) support functionals $\zeta^\theta$~\cite{strassen1991} agree for all weights $\theta$. %
    It is not known whether the quantum functionals and Strassen's upper support functionals are equal in general, but equality is known to hold for free tensors.
\end{remark}

\begin{remark}
The point~$\bigl(\diag(\tfrac{17}{42}, \tfrac{1}{3}, \tfrac{11}{42}), \diag(\tfrac{17}{42}, \tfrac{1}{3}, \tfrac{11}{42}), \diag(\tfrac{5}{14}, \tfrac{5}{14}, \tfrac{2}{7})\bigr)$ can also be realized as the moment map image of a tensor with free support, namely
\[
    T
    = \sqrt{\frac{5}{14}} e_{1,1,1} 
    + \sqrt{\frac{1}{21}} e_{1,2,2} 
    + \sqrt{\frac{1}{21}} e_{2,1,2}
    + \sqrt{\frac{2}{7}} e_{2,2,3}
    + \sqrt{\frac{11}{42}} e_{3,3,2}.
\]
\end{remark}

\section{Explicit non-free \texorpdfstring{$n\times n \times n$}{n x n x n} tensors for all \texorpdfstring{$n \geq 3$}{n at least 3}}
\label{section:explicit-non-free-nnn}

We proved that $\nurmievT_2$ and $\nurmievT_5$, which are concise tensors in $\C^3 \ot \C^3 \ot \C^3$, are not free (\cref{thm:c333 t2 t5 not free}). In this section we will construct, for every $\gldim \geq 3$, a concise tensor $T \in \C^\gldim \ot \C^\gldim \ot \C^\gldim$ that is not free. The construction can be thought of as a generalization of $\nurmievT_2$.

For $n\geq 2$, define the support $\Gamma_\gldim \subseteq [\gldim]^3$ as the set of all triples $(i,j,k) \in [n]^3$ such that
\begin{equation}
    i + j =
    \begin{cases}
        \gldim + 1 & \text{ if } k \in [\gldim-1], \\
        \gldim & \text{ if } k = \gldim.
    \end{cases}
\end{equation}
In other words, $\Gamma_n = \{(i,n+1-i,k) \mid i \in [n], k \in [n-1]\} \cup \{(i,n-i,n) \mid i \in [n]\}$.
For instance, for $\gldim = 3$, we have
\[
    \Gamma_3 = \{ (1,3,1), (1,3,2), (1,2,3), (2,2,1), (2,2,2), (2,1,3), (3,1,1), (3,1,2) \}.
\]
This corresponds to the support pattern as depicted below, similar to that of $\nurmievT_2$:
\begin{equation*}
\left[\!
    \begin{array}{ccc|ccc|ccc}
    \grayzero & \grayzero & \grayzero & \grayzero & \grayzero & * & * & * & \grayzero \\
    \grayzero & \grayzero & * & * & * & \grayzero & \grayzero & \grayzero & \grayzero \\
    * & * & \grayzero & \grayzero & \grayzero & \grayzero & \grayzero & \grayzero & \grayzero
    \end{array}\!
\right]
\end{equation*}
When $\gldim = 4$, this is extended as
\begin{equation*}
\left[\!
    \begin{array}{cccc|cccc|cccc|cccc}
    \grayzero&\grayzero&\grayzero&\grayzero& \grayzero&\grayzero&\grayzero&\grayzero& \grayzero&\grayzero&\grayzero&*& *&*&*&\grayzero \\
    \grayzero&\grayzero&\grayzero&\grayzero& \grayzero&\grayzero&\grayzero&*& *&*&*&\grayzero& \grayzero&\grayzero&\grayzero&\grayzero \\
    \grayzero&\grayzero&\grayzero&*& *&*&*&\grayzero& \grayzero&\grayzero&\grayzero&\grayzero& \grayzero&\grayzero&\grayzero&\grayzero \\
    *&*&*&\grayzero& \grayzero&\grayzero&\grayzero&\grayzero& \grayzero&\grayzero&\grayzero&\grayzero& \grayzero&\grayzero&\grayzero&\grayzero
    \end{array}\!
\right]
\end{equation*}
If instead one slices tensor along the third index, then the first~$n-1$ matrices are anti-diagonal matrices, and the last slice is shifted up.
\begin{remark}
    For~$n=2$,~$\Gamma_n = \{ (1,1,2), (1,2,1), (2,1,1) \}$ is the support of the so-called~$W$-state.
    This support is in particular free. For $n\geq3$, however, $\Gamma_n$ is not free.
\end{remark}
\subsection{Moment polytope inequality and minimum-norm point}
\label{subsection:nonfree-tensor-polytope-inequality}

We prove an inequality on the moment polytope of any tensor whose support is contained in~$\Gamma_\gldim$.
Recall that~$u_n = \frac1n (1, \dotsc, 1) \in \R^n$.
We define~$h \in \R^n \times \R^n \times \R^n$ and~$c \in \R$ by
\begin{equation}
  \label{eq:non-free tensors h c defn}
  h_1 = h_2 = \Bigl(\frac{\gldim-1}{2}, \frac{\gldim-1}{2} - 1, \dotsc, -\frac{\gldim-1}{2}\Bigr), \, %
  h_3 = u_\gldim - e_\gldim,\, c = \frac1\gldim.
\end{equation}
\begin{lemma}
    \label{lem:non-free tensor inequality}
    Let $n\geq 2$ and $T \in \C^\gldim \ot \C^\gldim \ot \C^\gldim$. If $\supp(T) \subseteq \Gamma_n$, %
    then for every point $p = (p_1 \sep p_2 \sep p_3) \in \Delta(T)$, we have $\<p, h> \geq c$.
\end{lemma}

\begin{proof}
    Let $\Gamma_n^\downarrow = \{(i,j,k) \mid \exists (a,b,c) \in \Gamma_n, (i,j,k) \leq (a,b,c)\}$ be the downward closure of $\Gamma_n$ (for the pointwise order). %
    If $(i,j,k) \in \Gamma_n^\downarrow$, then $i+j \leq \gldim+1$ and $k \in [\gldim-1]$, or $i+j\leq\gldim$ and $k=\gldim$.
    For every $U \in \GL_{\uppertriangular}$, $\supp(U \cdot T) \subseteq \Gamma_n^\downarrow$: indeed, if~$A$ is an upper-triangular matrix, the~$i'$-th entry of $A e_i$ is zero for~$i' > i$, and~$\supp(U \cdot T)$ being in the downward-closure of~$\supp(T) \subseteq \Gamma_n$ is the right analog of this statement for higher-order tensors.
    Let $\Omega_n = \{(e_i \sep e_j \sep e_k) \mid (i,j,k) \in \Gamma_n^\downarrow\}$.
    By \cref{theorem:tensor moment polytopes franz}, there is a $U \in \GL_{\uppertriangular}$ such that $\Delta(T) = \cap_L \conv \supp (LU \cdot T)$, where the intersection goes over all $L \in \GL_{\lowertriangular}$. We have $\cap_L \conv \supp (LU \cdot T) \subseteq \conv \supp (U\cdot T) \subseteq \conv \Omega_n$, since~$\supp(U \cdot T) \subseteq \Gamma_n^\downarrow$.
    To prove $\<p,h> \geq c$ holds for every $p \in \Delta(T)$, it thus suffices to prove that $\<p,h> \geq c$ holds for every $p \in \Omega_{\gldim}$. We check the two cases.
    \begin{itemize}
        \item If $k \in [\gldim-1]$, then $\<(e_i \sep e_j \sep e_k), h> = \gldim - 1 - (i+j-2) + \frac1{\gldim} \geq \frac1{\gldim} = c$, as $i+j \leq \gldim+1$.
        \item If $k = \gldim$, then $\<(e_i \sep e_j \sep e_k), h> = \gldim - 1 - (i+j-2) + \frac1{\gldim} - 1 \geq \frac1{\gldim} = c$, as $i+j \leq \gldim$.
    \end{itemize}
    This proves the claim. 
\end{proof}

With~$h$ and $c$ as in~\cref{eq:non-free tensors h c defn}, we define $q \in \R^n \times \R^n \times \R^n$ by
\begin{equation}
    \label{eq:q-defn}
    q = (u_\gldim \sep u_\gldim \sep u_\gldim) + c \frac{h}{\norm{h}^2}.
\end{equation}
For any vector $v \in \R^d$ we say that $v$ is non-negative if $v_i \geq 0$ for all $i \in [d]$, we say that $v$ is positive if $v_i > 0$ for all $i \in [d]$, and we say that $v$ is non-increasing if $v_1 \geq v_2 \geq \cdots \geq v_d$.
\begin{lemma}
    \label{lem:q-very-basic-properties}
    Let~$n \geq 2$.
    \begin{enumerate}[label=\upshape(\arabic*)]
    \item\label{item:q-is-orthogonal-projection} The vector~$q = (q_1 \sep q_2 \sep q_3)$ is the orthogonal projection of $(u_\gldim \sep u_\gldim \sep u_\gldim)$ onto the hyperplane $\{ p \mid \<p, h> = c \}$.
    \item\label{item:q-very-basic-weyl-chamber} For every $i \in [3]$, the vector $q_i$ is non-negative and non-increasing, with~$\sum_{j=1}^n (q_i)_j = 1$.
    \end{enumerate}
\end{lemma}
\begin{proof} (1) 
    The vector $h$ is a normal vector for the hyperplane, so the orthogonal projection adds a scalar multiple of $h$. 
    We see that $\langle (u_\gldim \sep u_\gldim \sep u_\gldim), h\rangle = \langle u_\gldim, h_1 \rangle + \langle u_\gldim, h_2 \rangle + \langle u_\gldim, h_3 \rangle = 0$, since each~$h_i$ has entries summing to~$0$. Therefore, $q = (u_\gldim \sep u_\gldim \sep u_\gldim) + c \frac{h}{\norm{h}^2}$.

    (2) %
    For every $i \in [3]$, $q_i$ is sorted non-increasingly, because $u_\gldim$ and $h_i$ are, and $c = 1/\gldim \geq 0$.
    Therefore we just have to establish that the last entry of $q_i$ is non-negative.
    Note that $(h_i)_\gldim \geq - \norm{h_i} \geq - \norm{h}$.
    Therefore,
    \begin{equation}\label{eq:item2pr}
        (q_i)_\gldim = \frac1\gldim \Bigl(1 + \frac{(h_i)_\gldim}{\norm{h}^2} \Bigr) \geq \frac1\gldim \Bigl(1 - \frac{1}{\norm{h}} \Bigr).
    \end{equation}
    We also have
    \[
        \norm{h_1}^2 = \norm{h_2}^2 = \sum_{k=0}^{\gldim-1} \left(\frac{\gldim-1-2k}{2}\right)^2 = \frac{1}{12} \gldim (\gldim^2 - 1),
    \]
    and $\norm{h_3}^2 = (1/\gldim)^2 (\gldim-1) + ((1/\gldim) - 1)^2 = (\gldim-1)/\gldim$.
    Therefore, $\norm{h}^2 = \gldim (\gldim^2 - 1) / 6 + (\gldim-1) / \gldim$, and then from $n\geq2$ %
    it follows that $\norm{h} \geq 1$.
    Combining this with \eqref{eq:item2pr} gives $(q_i)_\gldim \geq 0$.
    Lastly, $\sum_{j=1}^n (q_i)_j = \sum_{j=1}^n (u_n)_j + (h_i)_j = 1 + 0 = 1$ for~$i \in [3]$.
\end{proof}
The non-negativity and normalization in~\cref{lem:q-very-basic-properties}\ref{item:q-very-basic-weyl-chamber} in particular imply that~$(q_1 \sep q_2 \sep q_3)$ has the potential to appear in the image of the moment map~$\mu$ (recall that we identify points in~$\R^n \times \R^n \times \R^n$ with triples of diagonal matrices).
These two lemmas will later be used through the following corollary, allowing us to verify the conditions of~\cref{thm:free-tensor-has-free-tensor-in-minimal-K-orbit}.
\begin{corollary}
    \label{cor:T-with-supp-Gamma-and-mu-q-is-minnorm}
    Let~$T \in \C^n \ot \C^n \ot \C^n \setminus \{0\}$ be such that~$\supp(T) \subseteq \Gamma_n$ and~$\mu(T) = q$. %
    Then~$q$ is the minimum-norm point of~$\Delta(T)$.
\end{corollary}
\begin{proof}
    First observe that, since~$q \in \weylchamber$ by \cref{lem:q-very-basic-properties} and~$\mu(T) = q$, $q \in \weylchamber \cap \mu(\overline{\G \cdot [T]}) = \Delta(T)$.
    
    Next, we will show that for every~$p \in \Delta(T)$, $\norm{p}^2 \geq \norm{q}^2$.
    We first compute
    \[
        \norm{q}^2 = \norm*{(u_n \sep u_n \sep u_n) + c \frac{h}{\norm{h}^2}}^2
        = \norm*{(u_n \sep u_n \sep u_n)}^2 + \frac{c^2}{\norm{h}^2}
        = \frac{3}{n} + \frac{c^2}{\norm{h}^2}.
    \]
    Now we aim to show that every~$p \in \Delta(T)$ satisfies~$\norm{p}^2 \geq \norm{q}^2$.
    Every~$p = (p_1 \sep p_2 \sep p_3) \in \Delta(T)$ satisfies~$\sum_{j=1}^n (p_i)_j = 1$ for~$i \in [3]$, i.e., $\langle p_i, u_n \rangle = \frac1n$.
    There exists a unique decomposition~$p_i = p_i^\top + p_i^\perp$ where~$p_i^\perp \perp u_n$ for every~$i \in [3]$ and~$p_i^\top = \langle p_i, u_n \rangle u_n / \norm{u_n}^2$; then
    \[
        \norm{p_i}^2 = \norm{p_i^\top}^2 + \norm{p_i^\perp}^2 = \frac{\abs{\langle p_i, u_n \rangle }^2}{\norm{u_n}^2} + \norm{p_i^\perp}^2 = \frac1n + \norm{p_i^\perp}^2.
    \]
    We also have the inequality~$\langle p, h \rangle \geq c$ for every~$p \in \Delta(T)$ (\cref{lem:non-free tensor inequality}). Since~$h = (h_1 \sep h_2 \sep h_3)$ is such that~$h_i \perp u_n$ for every~$i \in [3]$, we find
    \[
        c^2 \leq \abs{\langle p, h \rangle}^2 = \abs{\langle p^\perp, h \rangle}^2 \leq \norm{p^\perp}^2 \norm{h}^2,
    \]
    where the inequality is the Cauchy--Schwarz inequality.
    Therefore for every~$p \in \Delta(T)$,
    \[
        \norm{p}^2 = \sum_{i=1}^3 \norm{p_i}^2 = \norm{p^\perp}^2 + \frac{3}{n^2 \norm{u_n}^2} \geq \frac{c^2}{\norm{h}^2} + \frac{3}{n^2 \norm{u_n}^2}
        = \frac{c^2}{\norm{h}^2} + \frac{3}{n} = \norm{q}^2,
    \]
    establishing that~$q$ is the minimum-norm point of~$\Delta(T)$.
\end{proof}

\begin{remark}
    It is also possible to use Ness' theorem~(\cref{lem:ness-checking-minimality}) to show that the inequality~$\langle \cdot, h \rangle \geq c$ holds for~$\Delta(T)$, at least if one can show there exists~$S \in \G \cdot T$ with~$\supp(S) \subseteq \Gamma_n$ and~$\mu(S) = q$, where~$q$ is the minimum-norm point on the hyperplane~$\{ p \mid \langle p,h \rangle = c \}$.
    Indeed, this~$q$ would be such that~$\exp(t \mu(S)) \cdot S = \exp(t q) \cdot S = \exp(t \lambda) \cdot S$ for some~$\lambda \in \R$ (by essentially the same computation as above, showing that~$\langle (e_i,e_j,e_k), q \rangle$ is constant for~$(i,j,k) \in \Gamma_n$).
    \Cref{lem:ness-checking-minimality} implies that~$q$ is the minimum-norm point in~$\Delta(T)$, and hence the inequality~$\langle p, q \rangle \geq \langle q,q \rangle$ holds for every~$p \in \Delta(T)$.
    This inequality is in turn equivalent to~$\langle p, h \rangle \geq c$.
    The tensor we will construct in this section does in fact have support contained in~$\Gamma_n$ and moment map image equal to~$q$, so~$q$ is the minimum-norm point in the tensor's moment polytope.
\end{remark}

\subsection{Explicit construction of a non-free tensor}
\label{subsection:proof-main-theorem}

\subsubsection*{Proof overview}
Now we want to construct a tensor with support contained in $\Gamma_n$ such that the tensor is not free. We thus need to choose coefficients on this support.
In our construction we will keep track of these coefficients in the following way.
Given a matrix $W \in \C^{n \times (n-1)}$ and vector $a \in \C^{n-1}$, we define the tensor $T^{W,a}$ by $T^{W,a}_{n+1-i,i,k} = W_{i,k}$ for $i \in [n]$ and $k < n$, $T^{W,a}_{n-i,i,n} = a_i$ for $i \in [n-1]$, and $T_{i,j,k} = 0$ for all other $(i,j,k) \in [n]^3$. 
For example, for $n = 4$ we have
\begin{equation*}
  T^{W,a} =
  \left[\arraycolsep=3pt\def\arraystretch{0.8}
    \begin{array}{cccc|cccc|cccc|cccc}
      \grayzero&\grayzero&\grayzero&\grayzero& \grayzero&\grayzero&\grayzero&\grayzero& \grayzero&\grayzero&\grayzero&a_1& W_{11}&W_{12}&W_{13}&\grayzero \\
      \grayzero&\grayzero&\grayzero&\grayzero& \grayzero&\grayzero&\grayzero&a_2& W_{21}&W_{22}&W_{23}&\grayzero& \grayzero&\grayzero&\grayzero&\grayzero \\
      \grayzero&\grayzero&\grayzero&a_3& W_{31}&W_{32}&W_{33}&\grayzero& \grayzero&\grayzero&\grayzero&\grayzero& \grayzero&\grayzero&\grayzero&\grayzero \\
      W_{41}&W_{42}&W_{43}&\grayzero& \grayzero&\grayzero&\grayzero&\grayzero& \grayzero&\grayzero&\grayzero&\grayzero& \grayzero&\grayzero&\grayzero&\grayzero
    \end{array}
  \right].
\end{equation*}
Our strategy is to first choose $W$ and $a$ such that $\mu(T^{W,a}) = q$.
Requirements on the tensor $T^{W,a}$ such that $\mu(T^{W,a}) = q$ translate directly into requirements on $W$ and $a$. We will denote the set of suitable $W$ by $\mathcal W_n \subseteq \C^{n\times (n-1)}$. It turns out we can choose $a$ independently of $W$, so we will fix its value and write $T^W$ in place of $T^{W,a}$.
We will show that $\mathcal W_n$ is non-empty (essentially through an application of the Schur--Horn theorem).

Having established that $\mu(T^{W}) = q$, 
\cref{cor:T-with-supp-Gamma-and-mu-q-is-minnorm} will tell us that $\mu(T^{W})$ is the minimum-norm point in the moment polytope $\Delta(T^{W})$.
This allows us to invoke our freeness condition from \cref{thm:free-tensor-has-free-tensor-in-minimal-K-orbit} which says: if $T^{W}$ is free, then its unitary orbit contains a tensor with free support. 
It then remains to show that no tensor in the unitary orbit of~$T^{W}$ has free support. 

To show this, the approach is as follows.
We suppose~$T^{W}$ has a tensor $S$ with free support in its unitary orbit. Then we show that $S$ equals (up to local permutation and scaling, which do not affect the freeness of the support)~$T^{W u}$ for some unitary~$u \in \U_{n-1}$ (using \cref{lemma:tensor moment map equivariance,lem:free-implies-diagonal-moment-map}). However, it turns out that~$\mathcal{W}_n$ is closed under right-multiplication with unitary matrices, so we may without loss of generality assume $T^{W}$ has free support.
To finish, we then show that any~$W \in \mathcal{W}_n$ must have a row with at least 2 non-zero entries, which implies that~$T^{W}$ does not have free support.

In~\cref{subsection:non-free-01-tensor} we will show that the~$T^W$ are equivalent to the tensor from~\cref{thm:explicit non-free in balanced format}, whose coefficients are all in~$\{0,1\}$.

\subsubsection*{Definition of \texorpdfstring{$\mathcal{W}_n$}{W_n} and theorem statement}
We will now describe the requirements on $W$ and $a$.
Let~$b \in \R^n$ by taking
\begin{equation}
  \label{eq:non-free tensors a}
  b_j = \sum_{\ell=1}^j (q_2)_\ell - (q_1)_{n+1-\ell}.
\end{equation}
for $j \in [n]$.
Define~$w \in \R^n$ by taking
\[
  w_j = \sqrt{\frac{1-(q_3)_n}{n-1} - (q_2)_j + b_j}
\]
for $j \in [n]$
(these are square roots of non-negative numbers by \cref{lem:a properties}\ref{item:w-positive}).
Let~$\mathcal{W}_n$ be the set of all matrices~$W \in \C^{n \times (n-1)} $ satisfying:
\begin{enumerate}[label=\upshape(\arabic*)]
  \item $W^*W = \frac{1 - (q_3)_n}{n-1} I_{n-1}$.
  \item $WW^* = \frac{1-(q_3)_n}{n-1} I_n - w w^*$.
\end{enumerate}
Define $a \in \R^n$ by taking $a_j = \sqrt{b_j}$ for~$j \in [n]$ ($b_j$ is non-negative by \cref{lem:a properties}\ref{item:bj-nonnegative}).
Recall that we write $T^{W}$ as a shorthand for $T^{W,a}$.

\begin{theorem}
\label{thm:explicit-concise-nonfree-tensor}
Let~$n \geq 3$. Then the set~$\mathcal{W}_n$ is not empty, and for every $W \in \mathcal{W}_n$, the tensor~$T^W \in \C^n \ot \C^n \ot \C^n$ is non-free and concise.
\end{theorem}

\subsubsection*{Technical properties}

Before we begin with the proof proper, we first establish various properties of the~$b_j$'s which will be relevant for the proof. 
\begin{lemma}
  \label{lem:a properties}
  Let~$n \geq 2$.
  \begin{enumerate}[label=\upshape(\arabic*)]
  \item\label{item:bj-nonnegative} 
  For $j \in [n-1]$ we have $b_j > 0$, and $b_n = 0$. 
  \item\label{item:b-normalization} $\sum_{j=1}^{n-1} b_j = (q_3)_n$.
  \item\label{item:bj-gaps} Writing~$b_0 = 0$, we have~$(q_2)_j - b_j = (q_1)_{n+1-j} - b_{j-1}$ for~$j \in [n]$.
  \item \label{item:w-positive} 
  For~$j \in [n]$, $0 \leq (q_2)_j - b_j < \frac{1-(q_3)_n}{n-1}$.
  \end{enumerate}
\end{lemma}
\begin{proof}
  (1)
  Recall that~$q = (u_n \sep u_n \sep u_n) + c \frac{h}{\norm{h}^2}$
  where~$c = \frac1n$ and~$h = (h_1 \sep h_2 \sep h_3)$ is as defined in~\cref{eq:non-free tensors h c defn}.
  Therefore~$b_j \geq 0$ reduces to
  \begin{equation*}
    \sum_{\ell=1}^j (h_2)_\ell - (h_1)_{n+1-\ell} \geq 0
  \end{equation*}
  for~$j \in [n]$.
  This quickly follows from~$h_1 = h_2 = (\frac{n-1}{2}, \frac{n-1}{2} - 1, \dotsc, -\frac{n-1}{2})$: we have~$\sum_{\ell=1}^j (h_2)_\ell \geq 0$ and $\sum_{\ell=1}^j (h_2)_\ell = - \sum_{\ell=1}^j (h_1)_{n+1-\ell}$.
  The inequality also becomes strict when~$j \in [n-1]$.

  (2)
  Observe that
  \begin{align*}
    \sum_{j=1}^{n-1} b_j & = \frac{c}{\norm{h}^2} \sum_{j=1}^{n-1} \sum_{\ell=1}^j (h_2)_\ell - (h_1)_{n+1-\ell} \\
               & = \frac{c}{\norm{h}^2} \sum_{j=1}^{n-1} \sum_{\ell=1}^j \big(n - 1 - 2 (\ell-1)\big) \\
               & = \frac{c}{\norm{h}^2} \sum_{j=1}^{n-1} j (n-j)\\
               & = \frac{c}{\norm{h}^2} \cdot \frac{1}{6} (n+1) n (n-1) \\
               & = \frac{c}{\norm{h}^2} \Big(\norm{h}^2 - \frac{n-1}{n}\Big) \\
               & = \frac{1}{n} + \frac{c}{\norm{h}^2} (h_3)_n = (q_3)_n
  \end{align*}
  using that~$\norm{h}^2 = n (n^2 - 1)/6 + (n-1)/n$ and~$(h_3)_n = 1/n - 1$.

  (3)
  The equality
    $b_j - b_{j-1}  = (q_2)_j - (q_1)_{n+1-j}$
  follows directly from the definition of~$b_j$.

  Finally we prove the bounds on~$(q_2)_j - b_j$; first we evaluate
  \begin{align*}
    (q_2)_j - b_j & = \frac{1}{n} + \frac{1}{n} \frac{\frac{n-1}{2} - 2(j-1)}{\norm{h}^2} - \sum_{\ell=1}^j ((q_2)_\ell - (q_1)_{n+1-\ell}) \\
                          & = \frac{1}{n} + \frac{1}{n} \frac{\frac{n-1}{2} - 2(j-1)}{\norm{h}^2} - \frac{1}{n} \frac{1}{\norm{h}^2} j (n-j) \\
                          & = \frac{1}{n} + \frac{1}{n} \frac{\frac{n-1}{2} - 2(j-1) - j (n-j)}{\norm{h}^2}.
  \end{align*}

  (4) We now prove~$(q_2)_j - b_j \geq 0$.
  It suffices to have
  \[
    \frac{n-1}{2} - 2(j-1) - j (n-j) \geq -\norm{h}^2 = -\left(\frac{1}{6} (n+1)n(n-1) + \frac{n-1}{n}\right).
  \]
  The left-hand side is minimized for~$j = \frac{n}{2} + 1$ where it attains the value~$\frac{3}{4} - \frac{1}{4} (n+1)^2$.
  Rewriting the above shows that the inequality is equivalent to
  \[
    \frac{n^4}{6} - \frac{n^3}{4} - \frac{2n^2}{3} + \frac{3 n}{2} - 1 \geq 0.
  \]
  Equality holds for~$n=2$; moreover, the difference between its value for~$n=m+1$ and~$n=m$ is
  \[
    \frac{2m^3}{3} + \frac{m^2}{4} - \frac{17m}{12} + \frac{3}{4}
  \]
  which is non-negative for~$m \geq 2$, as~$\frac{2m^3}{3} \geq \frac{17m}{12}$. This proves~$(q_2)_j - b_j \geq 0$.

  Lastly we show $(q_2)_j - b_j < \frac{1-(q_3)_n}{n-1}$, completing the proof of (4).
  Recall that~$(q_3)_n = \frac{1}{n} + c \frac{(h_3)_n}{\norm{h}^2} = \frac{1}{n} + \frac{1/n - 1}{n \norm{h}^2}$.
  Therefore we must show
  \begin{align*}
    \frac{1}{n} + \frac{1}{n} \frac{\frac{n-1}{2} - 2(j-1) - j (n-j)}{\norm{h}^2}
    = (q_2)_j - b_j
    & <
    \frac{1 - (q_3)_n}{n-1} =
    \frac{1}{n-1} \left(1 - \frac{1}{n} - \frac{1/n - 1}{n \norm{h}^2}\right) \\
    & = \frac{1}{n-1} \left(\frac{n-1}{n} - \frac{1-n}{n^2 \norm{h}^2}\right) \\
    & = \frac{1}{n} + \frac{1}{n^2 \norm{h}^2}.
  \end{align*}
  or equivalently
  \[
  \frac{n-1}{2} - 2(j-1) - j (n-j)
  <
  \frac{1}{n}
  \]
  for~$j \in [n]$.
  The left-hand side is maximized for~$j=1$ where its value is
  \[
    \frac{n-1}{2} - (n-1) = -\frac{n-1}{2}
  \]
  which is strictly smaller than $\frac{1}{n}$, given that~$n \geq 2 > 0$. This concludes the estimate~$(q_2)_j - b_j < (1-(q_3)_n) / (n-1)$.
\end{proof}

\subsubsection*{Construction of \texorpdfstring{$W \in \mathcal W_n$}{W}}
To construct an element of $\mathcal W_n$, we first prove the following more abstract lemma, concerning the existence of matrices~$W$ with similar conditions as in the definition of~$\mathcal{W}_n$.
\begin{lemma}
  \label{lem:special-schur-horn}
  Let~$n \geq 2$, $\lambda \in \R_{\geq 0}$, $d \in \R^n$ such that~$0 \leq d_i \leq \lambda$ for all~$i \in [n]$, and~$\sum_{i=1}^n d_i = (n-1) \lambda$.
  Then there exists a matrix~$W\in \C^{n \times (n-1)}$ satisfying
  \begin{enumerate}[label=\upshape(\arabic*)]
    \item $W^*W = \lambda I_{n-1}$, and
    \item %
    $WW^* = \lambda I_n  - ww^*$ with~$w \in \R^n$ given by~$w_i = \sqrt{\lambda - d_i}$ for~$i \in [n]$.
  \end{enumerate}
  In particular,~(2) implies that~$WW^*$ has diagonal entries given by~$d$, and spectrum~$(\lambda, \dotsc, \lambda, 0)$ since~$WW^*$ and~$W^*W$ are isospectral.
\end{lemma}
\begin{proof}
  Let~$w \in \R^n$ be given by~$w_i = \sqrt{\lambda - d_i}$.
  Note that~$w \neq 0$: not all~$d_i$ can be~$\lambda$, as~$0 \leq d_i$ and~$\sum_{i=1}^n d_i = (n-1) \lambda$.
  Let~$H = \lambda I_n - w w^*$.
  Then~$H$ is Hermitian, and its~$i$-th diagonal entry is given by~$\lambda - (\lambda - d_i) = d_i$.
  Moreover~$\norm{w}^2 = n \lambda - \sum_{i=1}^n d_i = \lambda$, and so the eigenvalues of~$H$ are given by~$\lambda$ with multiplicity~$n-1$ and~$0$ with multiplicity~$1$.
  The kernel is spanned by~$w$.
  Let~$w^1, \dotsc, w^{n-1} \in \R^n$ be an orthonormal basis of the orthogonal complement of~$w$.
  Set~$W = \sqrt{\lambda} [w^1 \sep \dotsc \sep w^{n-1}]$.
  Then~$W^* W = \lambda I_{n-1}$ by virtue of orthogonality of the~$w^j$, and~$H = \lambda I_n - w w^* = \sum_{j=1}^{n-1} \lambda w^j (w^j)^* = WW^*$.
\end{proof}
\begin{remark}
  The existence of~$W$ as in~\cref{lem:special-schur-horn} is a special case of the Schur--Horn theorem~\cite{hornDoublyStochasticMatrices1954}.
  It states that there exists a~$n \times n$ Hermitian matrix~$H$ with eigenvalues~$\lambda_1 \geq \dotsc \geq \lambda_n$ and diagonal entries~$d_1 \geq \dotsc \geq d_n$ if and only if~$(d_1, \dotsc, d_n)$ is majorised by~$(\lambda_1, \dotsc, \lambda_n)$, that is
  \[
    \sum_{i=1}^k d_i \leq \sum_{i=1}^k \lambda_i, \quad \sum_{i=1}^n d_i = \sum_{i=1}^n \lambda_i.
  \]
  In our setting we need an~$H = WW^*$ with diagonal entries~$(d_1, \dotsc, d_n)$ and eigenvalues $(\lambda, \dotsc, \lambda, 0)$.
  Here, a solution can also be explicitly constructed.
  It is not known how to do this (efficiently) in general.
\end{remark}

\begin{lemma}%
  \label{claim:Wn-not-empty}
  The set $\mathcal{W}_n$ is not empty.
\end{lemma}
\begin{proof}
  This follows from~\cref{lem:a properties}\ref{item:w-positive} combined with~\cref{lem:special-schur-horn}, setting~$d_j = (q_2)_j - b_j$ and~$\lambda = \frac{1-(q_3)_n}{n-1}$.
\end{proof}

\subsubsection*{Achieving the minimum-norm point}
We now show that the assumption that~$W \in \mathcal{W}_n$, as well as the specific choice of the entries~$a_i$ in the last columns of~$T^W$, guarantees that~$\mu(T^W) = q$. 
\begin{lemma}%
  \label{claim:TW-mu-is-q}
  For every $W \in \mathcal{W}_n$, we have $\mu(T^W) = q$, and~$q$ is the minimum-norm point of~$\Delta(T^W)$.
\end{lemma}
\begin{proof}
  Observe first that~$\norm{T^W} = \norm{W}^2 + \sum_{i=1}^{n-1} \abs{a_i}^2 = 1-(q_3)_n + (q_3)_n = 1$, where the second inequality follows from~$\abs{a_i}^2 = b_i$ and~\cref{lem:a properties}\ref{item:b-normalization}.
  The support of~$T^W$ is contained in~$\Gamma_n$. 
  The definition of the moment map (\cref{eq:tensor-moment-map}) implies that
  \[
    \mu_3(T^W) = W^*W \oplus [(q_3)_n] = \diag(q_3).
  \]
  It remains to check that~$\mu_1(T^W) = \diag(q_1)$ and~$\mu_2(T^W) = \diag(q_2)$.
  The support structure and~\cref{eq:tensor-moment-map} already imply that~$\mu_1(T^W)$ and~$\mu_2(T^W)$ are diagonal matrices, so it remains to check that the diagonal entries are correct.
  To obtain~$\mu_2(T^W) = \diag(q_2)$, recall that its~$j$-th diagonal entry is the squared~$2$-norm of the~$j$-th row of~$T^W$, so
  \begin{align*}
    \mu_2(T^W)_{j,j} = \sum_{i,k=1}^n \abs{T^W_{i,j,k}}^2 & = \abs{a_j}^2 + \sum_{k=1}^{n-1} \abs{W_{j,k}}^2 \\
                                                          & = \abs{a_j}^2 + (WW^*)_{j,j} = (q_2)_j,
  \end{align*}
  by definition of~$WW^*$ and~$\abs{a_j}^2 = b_j$.

  Next, the~$i$-th diagonal entry of~$\mu_1(T^W)$ is the squared $2$-norm of the~$i$-th slice of~$T^W$, hence given by
  \begin{align*}
    \sum_{j,k=1}^n \abs{T^W_{i,j,k}}^2 & = \abs{a_{n-i}}^2 + \sum_{k=1}^{n-1} \abs{W_{(n+1-i),k}}^2 \\
                                       & = \abs{a_{n-i}}^2 + (WW^*)_{(n+1-i),(n+1-i)} \\
                                       & = (q_1)_i.
  \end{align*}
  where we note that we have set~$a_0 = 0$.
  For the last equality we used that
  \[
    (WW^*)_{(n+1-i),(n+1-i)} = (q_2)_{n+1-i} - \abs{a_{n+1-i}}^2 = (q_1)_i - \abs{a_{n-i}}^2
  \]
  by~\cref{lem:a properties}\ref{item:bj-gaps} and~$\abs{a_j}^2 = b_j$.

  Finally, since~$T^W$ has support contained in~$\Gamma_n$ and~$\mu(T^W) = q$, \cref{cor:T-with-supp-Gamma-and-mu-q-is-minnorm} implies that~$q$ is the minimum-norm point of~$\Delta(T^W)$.
\end{proof}

\subsubsection*{Non-freeness of the tensor}
We now prove that~$T^W$ is not free.

\begin{lemma}
  \label{claim:TW-not-free}
  For every $W \in \mathcal{W}_n$, the tensor $T^W$ is not free.
\end{lemma}
\begin{proof}
  Suppose~$T^W$ is free.
  As~$\mu(T^W) = q$ is the minimum-norm point of~$\Delta(T^W)$ by~\cref{claim:TW-mu-is-q},
  \cref{thm:free-tensor-has-free-tensor-in-minimal-K-orbit} implies that there exists~$k \in \U_n \times \U_n \times \U_n$ such that~$k \cdot T^W$ has free support.
  Therefore every component of~$\mu(k \cdot T^W)$ is diagonal by~\cref{lem:free-implies-diagonal-moment-map}.

  We now further simplify~$k$.
  As~$\mu$ is equivariant with respect to the unitary group action, the components of~$\mu(k \cdot T^W)$ and~$\mu(T^W)$ are isospectral; but both are tuples of diagonal matrices, hence the diagonal matrices are related by permutations.
  By additionally acting with some~$\sigma \in (S_n)^3 \subseteq \U_n^3$, we may assume~$\sigma k \cdot \mu(T^W) = \mu(\sigma k \cdot T^W) = \mu(T^W)$.
  We note here that local permutations do not affect freeness.
  In other words, without loss of generality we may assume that~$k$ is in the stabilizer of~$\mu(T^W)$.
  Now we use that~$\mu(T^W) = q$: both~$q_1$ and~$q_2$ are strictly decreasing, and~$q_3$ is of the form~$(\frac{1-(q_3)_n}{n-1}, \dotsc, \frac{1-(q_3)_n}{n-1}, (q_3)_n)$ with~$(q_3)_n < \frac1n$.
  This implies that the stabilizer is given by the subgroup~$\U_1^n \times \U_1^n \times (\U_{n-1} \times \U_1) \subseteq \U_n^3$.
  By further acting with an element of~$\U_1^n \times \U_1^n \times (\{I_{n-1}\} \times \U_1)$, we may assume~$k$ only acts with some unitary~$u \in \U_{n-1}$ on the first~$n-1$~columns of~$T^W$; this again does not affect the freeness of the support of~$T^W$.
  Then we have
  \[
    k \cdot T^W = T^{W u^\top}
  \]
  where~$W u^\top$ is the usual matrix product of~$W$ with the transpose of~$u$.
  Next, note that $W u^\top \in \mathcal{W}_n$ as~$\bar{u} W^* W u^\top = \frac{1 - (q_3)_n}{n-1} \bar{u} u^\top = \frac{1 - (q_3)_n}{n-1} I_{n-1} = W^* W$ and~$W u^\top \bar{u} W^* = W W^*$. Here~$\bar{u}$ denotes the entrywise complex conjugate of~$u$.
  Therefore we may as well assume that we initially chose~$W$ such that~$T^W$ has free support.

  The matrix $W$ contains at least one non-zero element in every column, as~$W^* W$ has non-zero diagonal entries.
  Since~$W$ has~$n-1 \geq 2$ columns, there exist~$i, i'$ such that~$W_{i,1} \neq 0$ and~$W_{i',2} \neq 0$.
  In order for~$T^W$ to have free support, it must be the case that every row of~$W$ contains at most~one non-zero element.
  Therefore~$i \neq i'$, and the~$i$-th and~$i'$-th rows of~$W$ have zero inner product.
  But every off-diagonal entry of~$WW^* = \frac{1 - (q_3)_n}{n-1} I_n - ww^*$ is non-zero by virtue of~$w$ being entrywise positive (\cref{lem:a properties}\ref{item:w-positive}); in particular, the inner product between the~$i$-th and~$i'$-th row cannot be zero.
  This is a contradiction, and we conclude that~$T^W$ is not free.
\end{proof}

We are now in a position to prove \cref{thm:explicit-concise-nonfree-tensor}.
\begin{proof}[Proof of~\cref{thm:explicit-concise-nonfree-tensor}.]
  There exists~$W \in \mathcal W_n$ by~\cref{claim:Wn-not-empty}.
  The tensor~$T^W$ is not free by~\cref{claim:TW-not-free}.
  The conciseness of~$T^W$ follows from~\cref{claim:TW-mu-is-q}, as~$q$ has no zero entries; combined with the definition of~$\mu$ (\cref{def:moment-map-tensors}) this implies that $T^W$ has full-rank flattenings, that is, $T^W$ is concise.
  This completes the proof.
\end{proof}

\subsection{A non-free 0/1 tensor}
\label{subsection:non-free-01-tensor}
In this subsection we finish the proof of~\cref{thm:explicit non-free in balanced format} by showing that the tensors~$T^W$ from~\cref{thm:explicit-concise-nonfree-tensor} is~$\GL$-equivalent to the~$0/1$-tensor described in~\cref{thm:explicit non-free in balanced format}, and that this also holds for generic tensors with the right support. 

Let~$V_n = \{ S \in \C^n \otimes \C^n \otimes \C^n \mid \supp(S) \subseteq \Gamma_n \}$, where~$\Gamma_n =\{(i,j,k)\in [n]^3 \mid k < n \text{ and } i+j=n+1, \text{ or } k = n \text{ and } i+j = n\}$ as before.
Define~$\mathrm{W}\colon V_n \to \C^{n \times (n-1)}$ by
\[
\mathrm{W}_{i,k}(S) = S_{(n+1-i),i,k}, \quad i \in [n], \, k \in [n-1],
\]
and~$\mathrm{a}\colon V_n \to \C^{n-1}$ by~$\mathrm{a}_{i}(S) = S_{(n-i),i,n}$ for~$i \in [n-1]$.
The notion is chosen in analogy with that for the construction of~$T^W$.
Let~$S_0 \in V_n$ be the (unique) tensor with~$\mathrm{a}_i(S_0) = 1$ for~$i \in [n-1]$, and
\[
    \mathrm{W}(S_0) = \begin{bmatrix} e_1^\top \\ \vdots \\ e_{n-1}^\top \\ e_1^\top + \dotsb + e_{n-1}^\top\end{bmatrix} = 
    \left[\!
    \begin{array}{c}
    I_{n-1} \\
    \hline
    1 \, \cdots \, 1 
    \end{array}\!
    \right].
\]
For example, for~$n=4$,
\[
    S_0 = \left[
    \begin{array}{cccc|cccc|cccc|cccc}
    \grayzero & \grayzero & \grayzero & \grayzero & \grayzero & \grayzero & \grayzero & \grayzero & \grayzero & \grayzero & \grayzero & 1 & 1 & \grayzero & \grayzero & \grayzero \\
    \grayzero & \grayzero & \grayzero & \grayzero & \grayzero & \grayzero & \grayzero & 1 & \grayzero & 1 & \grayzero & \grayzero & \grayzero & \grayzero & \grayzero & \grayzero \\
    \grayzero & \grayzero & \grayzero & 1 & \grayzero & \grayzero & 1 & \grayzero & \grayzero & \grayzero & \grayzero & \grayzero & \grayzero & \grayzero & \grayzero & \grayzero \\
    1 & 1 & 1 & \grayzero & \grayzero & \grayzero & \grayzero & \grayzero & \grayzero & \grayzero & \grayzero & \grayzero & \grayzero & \grayzero & \grayzero & \grayzero
    \end{array}
    \right].
\]
Note that~$S_0$ exactly has the support as described in~\cref{thm:explicit non-free in balanced format}.
\begin{lemma}
\label{lem:generic-Vn-tensor-is-equivalent-to-S0}
Let~$S \in V_n$ be such that~$\mathrm{a}_i(S) \neq 0$ for~$i\in [n-1]$, and any~$n-1$ rows of~$\mathrm{W}(S)$ are linearly independent.
Then~$S$ is~$\G$-equivalent to~$S_0$.
\end{lemma}
\begin{proof}
By rescaling the rows of~$S$ we may assume that~$\mathrm{a}_i(S) = 1$ for all~$i \in [n-1]$.
Next, the assumption that any~$n-1$ rows of~$\mathrm{W}(S)$ are linearly independent implies that there exists an invertible~$(n-1) \times (n-1)$ matrix~$M$ such that~$\mathrm{W}(S) M$ has rows~$e_1^\top$, \dots, $e_{n-1}^\top$, and~$\lambda_1 e_1^\top + \dotsb + \lambda_{n-1} e_{n-1}^\top$ for some~$\lambda_1, \dotsc, \lambda_{n-1} \in \C$.
The~$\lambda_i$ are all non-zero: if~$\lambda_i$ is zero, then~$\mathrm{W}(S)_1 M, \dotsc, \widehat{\mathrm{W}(S)_i M}, \dotsc, \mathrm{W}(S)_{n} M$ are linearly dependent, where~$\widehat{\cdot}$ indicates omission, and~$\mathrm{W}(S)_j$ refers to the~$j$-th row of~$\mathrm{W}(S)$.
Since~$M$ is invertible, $\mathrm{W}(S)_1, \dotsc, \widehat{\mathrm{W}(S)_i}, \dotsc, \mathrm{W}(S)_n$ would then be linearly dependent, which is a contradiction.

Now let~$S' = (I \otimes (M^\top \oplus [1]) \otimes I) S$; then~$S' \in V_n$, $\mathrm{a}_i(S') = \mathrm{a}_i(S) = 1$ for~$i \in [n-1]$, and~$\mathrm{W}(S') = \mathrm{W}(S) M$.
Let~$S''$ be obtained from~$S'$ by multiplying the~$j$-th column with~$\lambda_j^{-1}$ for each~$j \in [n-1]$.
Then~$\mathrm{W}(S'')$ has~$e_1^\top + \dotsb + e_{n-1}^\top$ as last row, and its~$j$-th row is~$\lambda_j^{-1} e_j$ for~$j \in [n-1]$.
The~$(n-i,i,n)$-entries of~$S''$ are still~$1$ for~$i \in [n-1]$.

Concretely, for~$n=4$, the situation is as follows:
\[
S'' = \left[
    \begin{array}{cccc|cccc|cccc|cccc}
    \grayzero & \grayzero & \grayzero & \grayzero & \grayzero & \grayzero & \grayzero & \grayzero & \grayzero & \grayzero & \grayzero & 1 & \lambda_1^{-1} & \grayzero & \grayzero & \grayzero \\
    \grayzero & \grayzero & \grayzero & \grayzero & \grayzero & \grayzero & \grayzero & 1 & \grayzero & \lambda_2^{-1} & \grayzero & \grayzero & \grayzero & \grayzero & \grayzero & \grayzero \\
    \grayzero & \grayzero & \grayzero & 1 & \grayzero & \grayzero & \lambda_3^{-1} & \grayzero & \grayzero & \grayzero & \grayzero & \grayzero & \grayzero & \grayzero & \grayzero & \grayzero \\
    1 & 1 & 1 & \grayzero & \grayzero & \grayzero & \grayzero & \grayzero & \grayzero & \grayzero & \grayzero & \grayzero & \grayzero & \grayzero & \grayzero & \grayzero
    \end{array}
    \right].
\]
We now alternate between rescaling slices~$2, \dotsc, n$ and rows~$n-2, \dotsc, 1$, with the scalars chosen such that the non-zero entries all become~$1$.
In the above example, we would rescale slices~$2, 3, 4$ by~$\lambda_3$, $\lambda_3 \lambda_2$, $\lambda_3 \lambda_2 \lambda_1$, and rows~$2,1$ by $\lambda_3^{-1}$, $\lambda_3 \lambda_2$ respectively.
This shows that~$S''$ is equivalent to~$S_0$, completing the proof.
\end{proof}

We now show that the tensor~$T^W$ constructed in~\cref{subsection:proof-main-theorem} satisfies this property:
\begin{lemma}
\label{lem:TW-is-equivalent-to-S0}
For any~$W \in \mathcal W_n$, the tensor~$T^W$ is such that~$\mathrm{a}_i(T^W) \neq 0$ and any~$n-1$ rows of~$\mathrm{W}(T^W)$ are linearly independent.
\end{lemma}
\begin{proof}
We have~$\mathrm{W}(T^W) = W$, and~$\mathrm{a}_i(T^W) = a_i$.
The~$a_i$ are such that~$\abs{a_i}^2 = b_i > 0$ for~$i \in [n-1]$ (\cref{lem:a properties}\ref{item:bj-nonnegative}).
Next, $W \in \mathcal W_n$ implies that
\[
    WW^* = \frac{1 - (q_3)_n}{n-1} I_n - ww^*
\]
with~$w_j = \sqrt{\frac{1 - (q_3)_n}{n-1} - (q_2)_j + b_j}$.
Linear independence of the rows of the matrix~$W^i$ obtained from~$W$ by excluding the~$i$-th row is equivalent to its invertibility, which is in turn equivalent to invertibility of~$W^i (W^i)^*$, which is a principal submatrix of~$WW^*$.
More precisely, let~$v^i$ be obtained from~$w$ by excluding the~$i$-th entry; then
\[
    W^i (W^i)^* = \frac{1 - (q_3)_n}{n-1} I_{n-1} - v^i (v^i)^*.
\]
To check that this is invertible, it suffices to check that~$\norm{v^i}^2 \neq \frac{1 - (q_3)_n}{n-1}$.
To see that this is the case, note that
\begin{align*} 
    \norm{v^i}^2 & = \sum_{j=1, j \neq i}^n \abs{w_j}^2 \\
    & = \sum_{j=1, j \neq i}^n \left(\frac{1-(q_3)_n}{n-1} - (q_2)_j + b_j\right) \\
    & = 1 - (q_3)_n - \sum_{j=1, j \neq i}^n \big((q_2)_j - b_j\big).
\end{align*}
As~$\sum_{j=1}^n (q_2)_j = 1$ and~$\sum_{j=1}^{n} b_j = (q_3)_n$ (\cref{lem:a properties}\ref{item:b-normalization} noting that~$b_n = 0$),
\begin{align*}
    \norm{v^i}^2 & = (q_2)_i - b_i. 
\end{align*}
Thus it remains to check that~$(q_2)_i - b_i \neq \frac{1-(q_3)_n}{n-1}$.
But this follows from~\cref{lem:a properties}\ref{item:w-positive}.
\end{proof}

We now put everything together, proving the main theorem of this paper:
\begin{proof}[Proof of~\cref{thm:explicit non-free in balanced format}]
The tensors~$T^W$ with~$W \in \mathcal W_n$ constructed in~\cref{subsection:proof-main-theorem} are non-free (\cref{thm:explicit-concise-nonfree-tensor}).
They are equivalent to~$S_0$ by~\cref{lem:generic-Vn-tensor-is-equivalent-to-S0} and~\cref{lem:TW-is-equivalent-to-S0}, hence~$S_0$ is not free.
Moreover a generic tensor~$S \in V_n$ is equivalent to~$S_0$ by~\cref{lem:generic-Vn-tensor-is-equivalent-to-S0}, since its assumptions are generically satisfied.
\end{proof}

\section*{Acknowledgements}
MvdB, VL, and MW acknowledge support by the European Research Council (ERC Grant Agreement No.~101040907).
MvdB also acknowledges financial support by the Dutch National Growth Fund (NGF), as part of the Quantum Delta NL visitor programme.
MC and HN acknowledge financial support from the European Research Council (ERC Grant Agreement No.~818761), VILLUM FONDEN via the QMATH Centre of Excellence (Grant No.~10059) and the Novo Nordisk Foundation (grant NNF20OC0059939 `Quantum for Life'). MC also thanks the National Center for Competence in Research SwissMAP of the Swiss National Science Foundation and the Section of Mathematics at the University of Geneva for their hospitality. Part of this work was completed while MC was Turing Chair for Quantum Software, associated to the QuSoft research center in Amsterdam, acknowledging financial support by the Dutch National Growth Fund (NGF), as part of the Quantum Delta NL visitor programme.
HN also acknowledges support by the European Union via a ERC grant (QInteract, Grant No.~101078107).
MW also acknowledges the Deutsche Forschungsgemeinschaft (DFG, German Research Foundation) under Germany's Excellence Strategy - EXC\ 2092\ CASA - 390781972, the BMBF (QuBRA, 13N16135; QuSol, 13N17173) and the Dutch Research Council (NWO grant OCENW.KLEIN.267).
JZ was supported by NWO Veni grant VI.Veni.212.284.
Views and opinions expressed are those of the author(s) only and do not necessarily reflect those of the European Union or the European Research Council Executive Agency. Neither the European Union nor the granting authority can be held responsible for them.

\bibliographystyle{alphaurl}
\bibliography{references}

\appendix
\section{General moment polytope criterion for non-freeness}
\label{section:general-moment-polytope-criterion-freeness}

We discuss here a general version of the criterion introduced in \cref{section:moment-polytope-criterion-freeness}.
For a general overview of the terminology used here, we refer to~\cite{wallachGeometricInvariantTheory2017,georgoulasMomentWeightInequalityHilbert2021}.
Let~$G$ be a connected complex reductive group, $K$ a maximal compact subgroup, and suppose~$\pi\colon G \to \GL(V)$ is a rational representation of~$G$ on a finite-dimensional complex vector space~$V$.
We will write~$g \cdot v := \pi(g)v$ for~$g \in G$ and~$v \in V$.

Assume~$V$ is endowed with a~$K$-invariant inner product.
We denote the Lie algebra of~$K$ (resp., $G$) by~$\Lie(K)$ (resp., $\Lie(G)$), and assume that~$\Lie(K)$ is endowed with a (real) inner product invariant under the adjoint action of~$K$.
Such an inner product can be extended to a complex inner product on~$\Lie(G)$ by using the decomposition~$\Lie(G) = \Lie(K) \oplus i \Lie(K)$, which we will denote by~$\<\cdot, \cdot>$.
The space of linear maps~$\Lie(G) \to \C$ is denoted by~$\Lie(G)^*$.

The action of~$G$ on~$V$ induces a Lie algebra action of~$\Lie(G)$ on~$V$, given by
\[
    X \star v := \partial_{t=0} (\exp(tX) \cdot v)
\]
for~$v \in V$ and~$X \in \Lie(G)$. Here~$\exp\colon \Lie(G) \to G$ is the exponential map for~$G$.

We define the moment map in this general setting as follows.
For~$v \in V \setminus \{0\}$, consider the Kempf--Ness function~$F_v\colon G \to \R$ defined by~$F_v(g) = \log \norm{\pi(g) v}$.
\begin{definition}[Moment map]
  \label{def:moment-map-general}
  The moment map~$\mu\colon V \setminus \{0\} \to i \Lie(K)$
  is defined as the unique element of~$i \Lie(K)$ satisfying for all~$X \in i \Lie(K)$
  \begin{equation*}
    \<\mu(v), X> = \partial_{t=0} F_v(\exp(tX)) = \frac{\<v, X \star v>}{\norm{v}^2}.
  \end{equation*}
\end{definition}
\begin{remark}
  As the inner product on~$V$ is~$K$-invariant, the Kempf--Ness function is also invariant: $F_v(kg) = F_v(g)$ for all~$k \in K$ and~$g \in G$.
  As a consequence, $\partial_{t=0} F_v(\exp(tX))$ vanishes for~$X \in \Lie(K) \subseteq \Lie(G)$.
\end{remark}
\begin{remark}
  The moment map as defined in~\cref{def:moment-map-general} yields a moment map in the sense of symplectic geometry for the~$K$-action on~$\Proj(V)$ (endowed with the Fubini--Study form), see~\cite[Cor.~1.2.1]{nessStratificationNullCone1984}.
\end{remark}

Let~$D_K \subseteq K$ be a maximal torus%
\footnote{Though maximal tori are conventionally denoted with~$T$, this clashes with standard notation for tensors, so we opt for~$D$ (motivated by diagonal matrices in the setting of~$\GL_n$).}
and let~$D \subseteq G$ be its complexification.
Denote by $\Omega(V) \subseteq \Lie(D)^*$ the set of weights of~$V$, and write~$V = \oplus_{\omega \in \Omega(V)} V_\omega$ for its weight decomposition.
\begin{definition}[Support]
  \label{def:general-support}
  Let~$v \in V$. If~$v = \sum_{\omega \in \Omega(V)} v_\omega$ is its weight decomposition, then we define the \emph{support} of~$v$ by
  \begin{equation*}
    \supp(v) = \{\omega \in \Omega(V) \mid v_\omega \neq 0 \}.
  \end{equation*}
\end{definition}
\begin{definition}[Free support]
  \label{def:general-free-support}
  A subset~$\Omega' \subseteq \Omega(V)$ is called \emph{free} if for any two~$\omega, \omega' \in \Omega'$, $\omega - \omega'$ is not a root of~$G$.
\end{definition}
The general analog to \cref{lem:free-implies-diagonal-moment-map} is as follows.
\begin{lemma}[{\cite[Lem.~7.1]{sjamaarConvexityPropertiesMoment1998}, \cite[Prop.~2.2]{franz2002}}]
  \label{lem:general-freeness-moment-map-diagonal}
  Let~$v \in V \setminus \{0\}$.
  If~$\supp(v)$ is free, then~$\mu(v) \in i\Lie(D_K) \subseteq i\Lie(K)$.
\end{lemma}
\begin{proof}
  Recall the decomposition~$\Lie(G) = \Lie(D) \oplus \bigoplus_{\alpha} \Lie(G)_\alpha$ where the~$\alpha \in \Lie(D)^*$ are the roots of~$G$, i.e.\ the weights of $\Lie(G)$ under the conjugation action by $G$.
  Then~$\mu(v) \in \Lie(D)$ if and only if for every root~$\alpha$ and~$X \in \Lie(G)_\alpha$, one has~$\langle \mu(v), X \rangle = 0$.
  Next, it suffices to check this condition for~$X \in \Lie(G)_\alpha \cap i \Lie(K)$, as the complex span of~$\Lie(G)_\alpha \cap i \Lie(K)$ is~$\Lie(G)_\alpha$.

  Write~$v = \sum_{\omega \in \supp(v)} v_\omega$ for the weight decomposition.
  Then for~$X \in i \Lie(K)$,
  \[
    \<\mu(v),X> = \frac{1}{\norm{v}^2} \langle v, X \star v \rangle = \frac{1}{\norm{v}^2} \sum_{\omega, \omega' \in \supp(v)} \langle v_\omega, X \star v_{\omega'} \rangle.
  \]
  If~$X \in \Lie(G)_\alpha$, we find that~$X\star v_{\omega'} \in V_{\omega' + \alpha}$. Since~$\supp(v)$ is assumed to be free it holds that~$\omega' + \alpha \not\in \supp(v)$. We conclude that~$\langle v_\omega, X \star v_{\omega'} \rangle = 0$, since the weight decomposition is orthogonal.
  It follows that $\<\mu(v),X> = 0$ for $X \in \Lie(G)_\alpha$ for all roots $\alpha$, and we are done.
\end{proof}
\begin{definition}[Free vector]
  \label{def:general-freeness}
  A vector~$v \in V$ is called \emph{free} if there exists~$g \in G$ such that~$\supp(g \cdot v)$ is free.
\end{definition}

We now state the main theorem of this section.
\begin{theorem}
  \label{thm:general-free-tensor-has-free-tensor-in-minimal-K-orbit}
  Suppose~$v \in V \setminus \{0\}$ is free, and let~$w \in \overline{G \cdot v} \setminus \{0\}$ be such that~$\norm{\mu(w)} = \min_{p \in \Delta(v)} \norm{p}$.
  Then there exists~$k \in K$ such that~$\supp(k \cdot w)$ is free.
\end{theorem}
We will now work towards the proof of~\cref{thm:general-free-tensor-has-free-tensor-in-minimal-K-orbit}.
\begin{theorem}[{\cite[Thm.~6.4]{georgoulasMomentWeightInequalityHilbert2021}}]
  \label{thm:general-GRS-gradflow-ODE-properties}
  Let $v \in V \setminus \{0\}$.
  Then there is a solution~$v(t) \in V$, where $t \geq 0$, to the ordinary differential equation
  \begin{equation}
  \label{eq:GRS-ODE-general}
    \frac{\diff}{\diff t} v(t) = - \mu(v(t)) \star v(t), \quad v(0) = v.
  \end{equation}
  It satisfies~$v(t) \in G \cdot T$ for all $t \geq 0$.
  Moreover,
  writing~$[w] \in \Proj(V)$ for the equivalence class in projective space of a non-zero vector~$w \in V \setminus \{0\}$,
  the limit $[v]_\infty := \lim_{t \to \infty} [v(t)] \in \Proj(V)$ exists, and $\norm{\mu([v]_\infty)} = \min_{p \in \Delta(v)} \norm{p}$.
\end{theorem}

We also use the following theorem, showing that the minimum-norm points in any $G$-orbit-closure lie in one $K$-orbit.
\begin{theorem}[{\cite[Thm.~6.5]{georgoulasMomentWeightInequalityHilbert2021}}]
    \label{thm:general-second-ness-uniqueness}
    Let $v \in V \setminus \{0\}$. Let~$w, w' \in \closure{G \cdot v} \setminus \{0\}$ be such that~$\norm{\mu(w)} = \norm{\mu(w')} = \min_{p \in \Delta(v)} \norm{p}$.
    Then~$[w] \in K \cdot [w']$.
\end{theorem}
\Cref{thm:general-second-ness-uniqueness} is an extension of a classical theorem of Ness~\cite[Thm.~6.2]{nessStratificationNullCone1984}, where it is additionally assumed that~$w, w' \in G \cdot v$.
For the non-free tensors we construct in this paper, it suffices to use the original theorem of Ness. However, \cref{thm:general-second-ness-uniqueness} allows us to prove a more general criterion for freeness of tensors that we will need later.

\begin{corollary}
    \label{cor:general-moment-map-norm-minimizers-unique-in-orbit-closure}
    Let $v \in V \setminus \{0\}$.
    If $w \in \closure{G \cdot v}$ is such that $\norm{\mu(w)} = \min_{p \in \Delta(v)} \norm{p}$, then $[w] \in K \cdot [v]_\infty$.
\end{corollary}
\begin{proof}
  This follows directly from~\cref{thm:general-GRS-gradflow-ODE-properties,thm:general-second-ness-uniqueness}.
\end{proof}

\begin{proof}[Proof of~\cref{thm:general-free-tensor-has-free-tensor-in-minimal-K-orbit}.]
    Assume without loss of generality that~$v$ has free support, otherwise replace it by a vector in its orbit with free support; this does not affect~$\overline{G \cdot v}$ or~$\Delta(v)$.
    Consider the solution $v(t)$ to the gradient flow and its projective limit $[v]_\infty$ from \cref{thm:general-GRS-gradflow-ODE-properties}.
    We claim that~$v(t)$ remains in the orbit of~$v$ under the action of the maximal algebraic torus~$D \subseteq G$.
    Since the support of a vector is invariant under the action of $D$, this implies that $v(t)$ has the same support as $v$ for any $t$.
    Hence if we write $[v]_\infty = [v']$, then the support of $v' \in \overline{D \cdot v}$ is contained in the support of $v$, and hence $v'$ is also a free tensor.
    \Cref{cor:general-moment-map-norm-minimizers-unique-in-orbit-closure} implies that $[v'] \in K \cdot [w]$.
    After appropriate rescaling of $v'$, which does not affect freeness of its support, this implies $v' \in K \cdot w$, and we are done.
    
    We now prove that~$v(t)$ is in the~$D$-orbit of~$v$.
    Let~$\momentmaptorusgeneral\colon V \setminus \{0\} \to \Lie(D)$ be the composition of~$\mu$ with the orthogonal projection~$\Lie(G) \to \Lie(D)$.
    Consider the ODE
    \[
      \frac{\diff}{\diff t} \tilde v(t) = - \momentmaptorusgeneral(\tilde v(t)) \star \tilde v(t), \quad \tilde \tilde v(0) = v.
    \]
    Then the solution~$\tilde v(t)$ is in the~$D$-orbit of~$v$.
    To see this, first observe that $t \mapsto \exp(\int_0^t -\momentmaptorusgeneral(\tilde v(s)) \, \mathrm{d}s) \cdot v$ is also a solution to the ODE, because 
    \begin{align*}
        & \frac{\diff}{\diff t} \left(\exp\left(\int_0^t -\momentmaptorusgeneral(\tilde v(s)) \, \mathrm{d}s\right) \cdot v\right) \\
        & = \left.\frac{\diff}{\diff t'}\right|_{t'=0} \left[ \exp\left(\int_{t}^{t+t'} -\momentmaptorusgeneral(\tilde v(s)) \, \mathrm{d}s \right) \cdot \left(\exp\left(\int_0^t -\momentmaptorusgeneral(\tilde v(s)) \, \mathrm{d}s\right) \cdot v\right) \right] \\
        & = - \momentmaptorusgeneral(\tilde v(t)) \star \left(\exp\left(\int_0^t -\momentmaptorusgeneral(\tilde v(s)) \, \mathrm{d}s\right) \cdot v\right)
    \end{align*}
    by virtue of~$\int_0^t -\momentmaptorusgeneral(\tilde v(s)) \, \mathrm{d}s$ being an element of~$i \Lie(D_K) \subseteq \Lie(D)$ for every~$t \geq 0$, so that its exponential is easy to differentiate.
    By uniqueness of solutions to ODEs\footnote{The moment map, its diagonal part, and the infinitesimal action are locally Lipschitz functions, so uniqueness follows from a standard application of Picard--Lindel\"of theory.} we find that 
    \begin{equation*}
      \tilde v(t) = \exp\left(\int_0^t -\momentmaptorus(\tilde v(s)) \, \mathrm{d}s\right) \cdot v.
    \end{equation*}
    Because $\exp\bigl(\int_0^t -\momentmaptorusgeneral(\tilde v(s)) \, \mathrm{d}s\bigr) $ is an invertible diagonal matrix for every $t \geq 0$, this implies that $\tilde v(t) \in D \cdot v$.
    This yields~$\mu(\tilde v(t)) = \momentmaptorusgeneral(\tilde v(t))$ by~\cref{lem:general-freeness-moment-map-diagonal}.
    Hence $\tilde v(t)$ is also a solution to the ODE of 
    \cref{thm:general-GRS-gradflow-ODE-properties} (\cref{eq:GRS-ODE-general}).
    Another application of uniqueness of solutions to ODEs yields~$\tilde v(t) = v(t)$ for all~$t \geq 0$, and we conclude~$v(t) = \tilde v(t) \in D \cdot v$.
\end{proof}

\end{document}